\documentclass{amsart}%
\usepackage{amsfonts}
\usepackage{amsmath}
\usepackage{amssymb}
\usepackage{graphicx}
\usepackage{xcolor}
\usepackage[colorlinks, citecolor = blue, pagebackref]{hyperref}%
\theoremstyle{plain}
\newtheorem{theorem}{Theorem}
\theoremstyle{plain}

\theoremstyle{plain}
\newtheorem{lemma}{Lemma}
\theoremstyle{plain}
\newtheorem{cor}{Corollary}
\theoremstyle{definition}
\newtheorem{defin}{Definition}
\theoremstyle{remark}
\newtheorem{remark}{Remark}
\theoremstyle{remark}

\theoremstyle{remark}

\begin{document}

\Large
\title[Random unconditional convergence of Rademacher chaos in $L_\infty$.]{Random unconditional convergence of Rademacher chaos in $L_\infty$ and sharp estimates for discrepancy of weighted graphs and hypergraphs.\\}
\author[S.V. Astashkin]{Sergey V. Astashkin}
\address[S.V. Astashkin]{Department of Mathematics, Samara National Research University, Moskovskoye shosse 34, 443086, Samara, Russia; Lomonosov Moscow State University, Moscow, Russia; Moscow Center of Fundamental and Applied Mathematics, Moscow, Russia; Department of Mathematics, Bahcesehir University, 34353, Istanbul, Turkey.}
\email{astash56@mail.ru}
\urladdr{www.mathnet.ru/rus/person/8713}

\author[K.V. Lykov]{Konstantin V. Lykov}
\address[K.V. Lykov]{Institute of Mathematics of the National Academy of Sciences of Belarus, Minsk, Belarus; Belarusian State University, Minsk, Belarus}
\email{alkv@list.ru}
\thanks{The work of the first named author was performed at Lomonosov Moscow
State University and was supported by the Russian Science
Foundation, project no. 23-71-30001. The work of the second named author was supported by the State Research Programme ''Convergence-2025''\;of the National Academy of Sciences of Belarus. }
\date{}
\subjclass{Primary 46B09; Secondary 05C15, 05C35,46E30}
\keywords{graph, hypergraph, discrepancy, edge-weighted graph, edge-weighted hypergraph, Rademacher functions, Rademacher chaos, multiple Rademacher system, random unconditional convergence, matrix norm, cut-norm, random matrix.}

\begin{abstract}
We prove that both multiple Rademacher system and Rademacher chaos possess the property of random unconditional convergence in the space $L_\infty$. This fact combined with some intimate connections between $L_\infty$-norms of linear combinations of elements of these systems and some special norms of matrices of their coefficients allows us to establish sharp two-sided estimates for the discrepancy of edge-weighted graphs and hypergraphs. Some of these results extend the classical theorem proved by Erd\"os and Spencer for the unweighted case.
\end{abstract}
\maketitle

%\begin{center}
%{\bf Random unconditional convergence of Rademacher chaos in $L_\infty$ and sharp estimates for discrepancy of weighted graphs and hyper graphs.\\[5mm]}
%\end{center}

%\begin{center}
%{\bf С.В. Асташкин\footnote[1]{Работа выполнена в рамках реализации программы развития Научно-образовательного математического центра Приволжского федерального округа, соглашение № 075-02-2023-931.}, К.В. Лыков\footnote[1]{Работа была поддержана НАН Беларуси в рамках ГПНИ “Конвергенция-2025” }\\[5mm]}
%\end{center}

%{\bf Ключевые слова}: функции Радемахера, хаос Радемахера, случайная безусловная сходимость, уклонение, граф, взвешенный гиперграф, матричные нормы, случайные матрицы.

%{\bf Key words}: Rademacher functions, Rademacher chaos, random unconditional convergence, discrepancy, graph, weighted hypergraph, matrix norms, random matrices.

%{\bf Аннотация}: Исследуются свойства 

%\begin{center}
%Асташкин С.В.\footnote{Самарский национальный исследовательский университет имени академика С. П. Королева, Московское шоссе, 34, 443086 Самара, Россия}, Лыков К.В.
%\footnote{
%Институт систем обработки изображений РАН - филиал Федерального государственного учреждения "Федеральный научно-исследовательский центр "Кристаллография и фотоника"\;Российской академии наук"\,, ул. Молодогвардейская, 151, 443001 Самара, Россия; 
%e-mail: alkv@list.ru}
%}
%\end{center}

\section{Introduction}
\label{Intro}

The discrepancy (or imbalance) of a hypergraph $H=(V,E)$ ($V$ is the vertex set and $E$ is the edge set) is defined by the formula
%{\color{red}
%$$
%$\mathrm{disc}(H_{n,d})=\min_{\theta}\max_{V'\subset V}\Bigl|\sum_{e\in E,\,e\subset V'}\theta(e)\Bigr |,
%$$
%}
%{\color{blue}
$$
\mathrm{disc}(H)=\min_{\theta}\max_{V'\subset V}\Bigl|\sum_{e\in E,\,e\subset V'}\theta(e)\Bigr |,
$$
where the minimum is taken over all functions (colorings) $\theta:\,E\to\{-1,1\}$.
In 1971, Erd\"os and Spencer proved for the complete $d$-homogeneous  hypergraph $H_{n,d}$ with $n$ vertices ($d\le n$) the following estimates:
%have introduced the notion of discrepancy (or imbalance) of the complete $d$-homogeneous  hypergraph $H_{n,d}=(V,E)$, $d\le n$, with $n$ vertices by the formula
%$$
%\mathrm{disc}(H_{n,d})=\min_{\theta}\max_{V'\subset V}\Bigl|\sum_{e\in E,\,e\subset V'}\theta(e)\Bigr |,
%$$
%where the minimum is taking over all functions (colorings) $\theta:\,E\to\{-1,1\}$. In the same classical paper \cite{ErSp}, they proved for this quantity the following estimates:
\begin{equation}
\label{Erd-Sp intr}
c_d n^{\frac{d+1}{2}}\leqslant \mathrm{disc}(H_{n,d})\leqslant C_d n^{\frac{d+1}{2}},
\end{equation}
with some constants $c_d$ and $C_d$ independent of $n$ \cite{ErSp}. This result marked the emergence of a new branch of the extremal graph theory, which is concerned with estimation of the discrepancy and its analogues for graphs and hypergraphs and which is rapidly developing at the present time (see, for instance,  \cite{AlonSpencer} or \cite{RCh}).

In this paper, among other results, we extend estimates \eqref{Erd-Sp intr} to a more general setting of weighted hypergraphs. It is worth to note that this generalization can be useful in solving various problems of Graph theory, because such weights might represent, for example, costs, lengths, capacities, depending on the problem at hand.

Let us assign to each edge $e\in E$ of a hypergraph $H=(V,E)$ a weight $w(e)\in\mathbb{R}$. The discrepancy of the resulting weighted hypergraph $H(W)$, where $W=\{w(e)\}_{e\in E}$, is defined as follows
$$
\mathrm{disc}(H(W))=\min_{\theta}\max_{V'\subset V}\Bigl|\sum_{e\in E,\,e\subset V'}\theta(e)w(e)\Bigr|.
$$
Here, we show that for every $d\in\mathbb{N}$ there exist constants $c_d'$ and $C_d'$ such that for each $d$-homogeneous hypergraph $H=(V,E)$ and any set of weights $W$ it holds
% with $n$ vertices, $2\le d\le n$, there exist constants $c_d'$ and $C_d'$ independent of $n$ and the set of weights $W$ such that
\begin{equation}
\label{extension of Erd-Sp intr}
c_d' \cdot \sum_{v\in V}\Bigl(\sum_{e\in E:\,v\in e}w(e)^2\Bigr)^{1/2}\leqslant \mathrm{disc}(H(W))\leqslant  C_d' \cdot\sum_{v\in V}\Bigl(\sum_{e\in E:\,v\in e}w(e)^2\Bigr)^{1/2}
\end{equation}
(see Theorem \ref{theor_weighted_hypergraph} below). Observe that in the special case when $H=H_{n,d}$ and $w(e)\equiv 1$ these estimates imply the above inequalities 
\eqref{Erd-Sp intr}. Moreover, discrepancy $\mathrm{disc}(H(W))$ is equivalent (up to constants depending only on $d$) to the expectation 
$$
\mathsf{E}_\theta  \max_{V'\subset V}\Bigl|\sum_{e\in E,\,e\subset V'}\theta(e)w(e)\Bigr|,$$
which is taken over all colorings $\theta$.

%{\color{blue} Then there exist constants $c_d'$ and $C_d'$, independent of $n$ and the set of weights $W$, such that
%\begin{equation}
%\label{extension of Erd-Sp intr}
%c_d' \cdot \sum_{v\in V}\Bigl(\sum_{e\in E:\,v\in e}w(e)^2\Bigr)^{1/2}\leqslant \mathrm{disc}(H(W))\leqslant  C_d' \cdot\sum_{v\in V}\Bigl(\sum_{e\in E:\,v\in e}w(e)^2\Bigr)^{1/2}
%\end{equation}
%(see Theorem \ref{theor_weighted_hypergraph} below). Observe that in the special case when $H=H_{n,d}$ and $w(e)\equiv 1$ these estimates imply the above inequalities 
%\eqref{Erd-Sp intr}. Observe additionally that discrepancy $\mathrm{disc}(H(W))$ is equivalent (up to universal constants) to the expectation 
%$$
%\mathsf{E}_\theta  \max_{V'\subset V}\Bigl|\sum_{e\in E,\,e\subset V'}\theta(e)w(e)\Bigr|$$
%over all colorings $\theta$.}

The main tools we use to prove equivalence \eqref{extension of Erd-Sp intr} and some other results are in fact beyond Graph theory.  Being mostly of functional-analytic and probabilistic nature, they are interesting we hope in their own right. More precisely, these tools are related to the specific properties of systems built on the basis of the classical Rademacher system (or, in other words, a sequence of independent, identically and symmetrically distributed random variables taking on values $\pm 1$). Let us describe briefly them. 

First of all, of key importance to us will be the fact that both multiple Rademacher system and Rademacher chaos possess the property of random unconditional convergence in the space $L_\infty$.
%\footnote{This concept was introduced by Billard, Kwapie\'n, Pe{\l}czy\'nski, and Samuel in the paper \cite{BKPS}.}.
Namely, focusing on the case of the second-order Rademacher chaos $\{r_ir_j\}_{1\le i<j\le n}$, for all $n\in\mathbb{N}$ and $a_{i,j}\in\mathbb{R}$, $1\leqslant i<j\leqslant n$, with universal constants we have
\begin{eqnarray}
\mathsf{E}\,_\theta\Bigl\|\sum_{i=1}^n\sum_{j=i+1}^n \theta_{i,j}a_{i,j}r_i r_j\Bigr\|_{L_\infty([0,1])}&\asymp&
\min_{\theta_{i,j}=\pm 1}{\Big\|\sum_{i=1}^n\sum_{j=i+1}^n \theta_{i,j}a_{i,j}r_i r_j\Big\|}_{L_\infty([0,1])}\nonumber\\
&\asymp&\max\Big\{\sum_{i=1}^{n-1}\Big(\sum_{j=i+1}^n a_{i,j}^2\Big)^{1/2},\sum_{j=2}^n\Big(\sum_{i=1}^{j-1} a_{i,j}^2\Big)^{1/2}\Big\}, 
\label{ RUC iin intr}
\end{eqnarray}
where by $\mathsf{E}\,_\theta$ is meant the expectation over all arrangements of signs $\theta_{i,j}=\pm 1$, $1\le i<j\le n$ (see Theorem \ref{theor_chaos}). Recall that the Rademacher functions $r_i(t)$, $t\in[0,1]$, can be defined by the formula
$$
r_i(t):=(-1)^{[2^i t]},\quad i=1,2,\dots,
$$
where $[x]$ denotes the integer part of a real number $x$. 
%In other words, $\{r_i\}_{i=1}^\infty$ is a sequence of independent, symmetrically distributed random variables taking values $\pm 1$.

It is worth to note that in the unimodular case, that is, when $a_{i,j}= \pm 1$ equivalence \eqref{ RUC iin intr} has been known for a long time. In 1977, Bennett considered the set $\mathcal{A}$ of matrices $A=(a_{i,j})_{1\le i\le n,1\le j\le m}$, $a_{i,j}= \pm 1$, treating them as random operators acting from $\ell_p^m$ into $\ell_q^n$.
%\{\color{blue} studying the set $\mathcal{A}$, consisting of matrices of the form $A=(a_{i,j})_{1\le i\le n,1\le j\le m}$, $a_{i,j}= \pm 1$, interpreted as random operators}, 
He showed that, for every $1\le p,q\le\infty$, it holds
%{\color{red}$$
%\mathsf{E}_A\|A:\,\ell_p^m\to \ell_q^n\|\asymp \min_{A\in \mathcal{A}}\|A:\,\ell_p^m\to \ell_q^n\|$$}
%{\color{blue}
$$
\mathsf{E}_\mathcal{A}\|A:\,\ell_p^m\to \ell_q^n\|\asymp \min_{A\in \mathcal{A}}\|A:\,\ell_p^m\to \ell_q^n\|\asymp\max\bigl\{n^{\frac{1}{q}}m^{\max\{\frac{1}{2}-\frac{1}{p},0\}},m^{\frac{p-1}{p}}n^{\max\{\frac{1}{q}-\frac{1}{2},0\}}\bigr\}$$
with constants that depend only on $p$ and $q$ \cite{Bennett} (see also \cite{BGN-75}). In view of the representation of the norm 
$$
\Big\|\sum_{i=1}^n\sum_{j=1}^m a_{i,j}r_i\otimes r_j\Big\|_{L_\infty([0,1]^2)}$$ (here, $(r_i\otimes r_j)(u,v)=r_i(u)r_j(v), (u,v)\in [0,1]^2$) as the norm $\|A:\,\ell_\infty^m\to \ell_1^n\|$ of the operator, generated by the matrix $A=(a_{i,j})_{i\leqslant n, j \leqslant m}$ (see \eqref{eq_operator_norm} below), the latter equivalence combined with a decoupling argument (see Subsect. \ref{decoupling} below) implies estimates \eqref{ RUC iin intr} in this partial case. Later on, in other terms, the same result has been rediscovered in \cite[Theorems 5 and 6]{As1998}.

The second tool, which is also crucial for us, is the existence of intimate connections between $L_\infty$-norms of linear combinations of elements of the multiple Rademacher system (resp. of the Rademacher chaos) and some special norms of matrices of their coefficients (see Subsect.  \ref{Sec_matrix_norm}). This fact reveals, in particular, interesting geometric properties of the subspaces generated by these  systems in $L_\infty$.

Let us describe briefly the structure of the paper. In Sect. \ref{Preliminaries} we gather some definitions and auxiliary results we will use in the sequel. In particular, in Subsect. \ref{decoupling}, we recall the decoupling results that will allow us to transfer estimates, obtained first for the multiple Rademacher system, to the Rademacher chaos. In the next subsection, as was said above, we introduce some special matrix norms and show their connection with $L_\infty$-norms of Rademacher sums, corresponding to considered systems.

%We show here also that considered Raemacher sums are determined (up to equivalence) by special matrix norms.

Sections \ref{Sec_mult_Rad} --- \ref{Sec_mult_d_Rad} contain the main results of the paper. In the first two sections, we prove that both second-order multiple Rademacher system and Rademacher chaos have in $L_\infty$ the property of random unconditional convergence. Then, in Sect. \ref{Sec_mult_d_Rad}, similar results are obtained for these  systems of an arbitrary order.

Finally, in Sect. \ref{Sec_disc_edges_graph}, we deal with applications of the results obtained to estimation of the discrepancy of edge-weighted graphs and hypergraphs. In particular, we are able to find sharp (up to equivalence) two-sided estimates for any such graphs and  homogeneous hypergraphs (see \eqref{extension of Erd-Sp intr}).
 
In conclusion, we would like to say a few words about another possible application of the property of random unconditional convergence of the Rademacher chaos  in the space $L_\infty$
(which is not reflected however in this paper).

As is known (see \cite{Talagrand, Talagrand2}), the exchange interaction energy of magnetic ions in the Sherrington-Kirkpatrick spin glass model is described by the  formula 
\begin{equation}\label{spin_glasses}
E=\sum_{i<j} g_{i,j}\sigma_i\sigma_j,
\end{equation}
where the coefficients $g_{i,j}$ are random (this is caused by the non-uniform distribution of impurities in the material and the behaviour of conduction electrons) and $\sigma_i\in\{-1,1\}$. For fixed values of the coefficients $g_{i,j}$, we are interested in the minimum of the energy $E$ with respect to all possible  $\sigma_i=\pm 1$, as well as in identification of specific configurations of the above signs at which this minimum is achieved. The problem of finding the corresponding extreme values, which is equivalent to that of quadratic constrained binary optimization (QUBO), is both difficult and important. In particular, many discrete optimization problems arising in applications (e.g., the travelling salesman problem) are reduced to it by means of 0/1 integer programming \cite{Karp, KHG, ABC}. Note that this problem is partially solved physically by using an adiabatic quantum computer \cite{VMKG, ZBDE, GKH}. However, theoretical results on the extreme values of the energy $E$ (see \eqref{spin_glasses}) are still very useful, because they allow to optimize  expensive quantum computations. Coming back finally to the content of this paper, observe that the problem of finding the maximum of the modulus of quantity \eqref{spin_glasses} over all $\sigma_i$ 
(closely related to the above-mentioned one of minimization of $E$) coincides with that of computing the $L_\infty$-norm of the second-order Rademacher chaos.

\section{Preliminaries and some auxiliary resuts.}
\label{Preliminaries}

In what follows, any expression of the form $F_1\asymp F_2$ means that $cF_1\leqslant F_2\leqslant CF_1$ for some constants $c>0$ and $C>0$, which do not depend on all or a part of the arguments of the functions (quasi-norms) $F_1$ and $F_2$. %and it should be clear from the context which arguments are being discussed.
%Also, the sign $|\cdot|$ denotes the modulus of a number or a function, as well as the cardinality of a set, depending on the context. 
Also, as usual, $[n]:=\{1,2,\dots,n\}$ for $n\in\mathbb{N}$, $L_p:=L_p[0,1]$ for $1\le p\le\infty$ and $\|a\|_2:=(\sum_{j=1}^\infty a_j^2)^{1/2}$.
%Finally, for a sequence $a=(a_j)_{j=1}^\infty\in\ell_2$ we put $\|a\|_2:=(\sum_{j=1}^\infty a_j^2)^{1/2}$.

\subsection{Khintchine's inequality}\label{Sec_Khintchine}

According to the classical Khintchine inequality \cite{khint} (see also \cite[Theorem~6.2.3]{AK} or \cite[Theorem 1.3]{AsBook}), for every $1\le p< \infty$ there exists a constant $C_p$ such that for any sequence $a=(a_j)_{j=1}^\infty\in\ell_2$ we have
\begin{equation}
\label{basic}
\Big\|\sum_{j=1}^\infty
a_jr_j\Big\|_{L_p}\leqslant C_p{\|a\|}_2.
\end{equation}
It is well known that $C_p\leqslant \sqrt{p}$ (the exact value of this constant for every $p$ was found by  Haagerup \cite{Haag}). Moreover, in \cite{Szarek}, Szarek proved that for all $p\geqslant 1$ and $a_k\in\mathbb{R}$, $k=1,2,\dots$, 
\begin{equation}
\label{basic1}
\frac{1}{\sqrt{2}}{\|a\|}_2\leqslant\Big\|\sum_{j=1}^\infty
a_jr_j\Big\|_{L_p},
\end{equation}
and in the case $p=1$ the constant $1/{\sqrt{2}}$ is exact. 

%{\color{red} 
These inequalities have given rise to a large number of investigations finding numerous applications in various areas of analysis. 
%It is known that Khintchine proved inequality \eqref{basic}, ''pursuing the goal of clarifying the correct rate of convergence in E. Borel's strong law of large numbers'' (see  \cite{PeshShir}). At the same time, 
In particular, inequalities \eqref{basic} and \eqref{basic1} show that each of the spaces $L_p[0,1]$, $1\leqslant p<\infty$, although not being Hilbert for $p\neq 2$, contains a subspace isomorphic to $\ell_2$. On the other hand, by the definition of the Rademacher functions, for every $n\in\mathbb{N}$ and $\theta_k=\pm 1$, $k=1,\dots,n$, there is $t\in [0,1]$ such that $r_k(t)=\theta_k$, $k=1,\dots,n$. Hence, we see that the $L_p$-norms of Rademacer sums "make a jump"\:when $p\to\infty$, because for every $a_k\in\mathbb{R}$ it holds
\begin{equation*}
\Big\|\sum_{k=1}^n a_kr_k\Big\|_{L_\infty}=\sum_{k=1}^n|a_k|,\;\;n\in\mathbb{N}.
\end{equation*}
This fact has motivated, in particular, investigations of the behaviour of Rademacher sums in the more general setting of rearrangement invariant spaces\footnote{For the definition of rearrangement invariant spaces and detailed information related to them we refer to the books \cite{LT} and \cite{KPS}.}. 
The problem of finding conditions, under which the sequence $\{r_j\}_{j=1}^\infty$ is equivalent in an rearrangement invariant space $X$ to the canonical basis in $\ell_2$, was resolved by Rodin and Semenov \cite{RS} (see also \cite[Theorem~2.b.4]{LT} or \cite[Theorem~2.3]{AsBook}). Namely, they proved that this holds if and only if $X$ contains the separable part of the exponential Orlicz space $\mathrm{Exp}L^{2}$ generated by the function $N_2(u)=e^{u^{2}}-1$\footnote{This space consists of all measurable functions $x(u)$ on $[0,1]$ such that $\lim_{t\to +0}x^*(t)\log^{-1/2}(2/t)=0$, where $x^*(t)$ is the nondecreasing rearrangement of $|x(u)|$.}.
% В работе \cite{As1998} аналогичный вопрос изучался для системы $\{r_{j_1}(t)\cdot r_{j_2}(t)\}_{j_1>j_2}$ произведений функций Радемахера, именуемой хаосом Радемахера второго порядка. Там было показано, что эта система эквивалентна в $X$ каноническому базису в $\ell_2$ тогда и только тогда, когда $X$ decreasing rearrangement of $|x|$содержит сепарабельную часть пространства Орлича $\mathrm{Exp}L$, построенного по функции $N_1(u)=e^u-1$. Кроме того, оба последних свойства оказались равносильными формально более слабому (чем эквивалентность каноническому базису в $\ell_2$)  свойству безусловной базисности системы $\{r_{j_1}(t)\cdot r_{j_2}(t)\}_{j_1>j_2}$ в $X$ \cite{As2000}. Отметим, что сама система Радемахера является безусловной (и даже симметричной с константой 1) базисной последовательностью в любом симметричном пространстве \cite[предложение~2.2]{AsBook}. 

\subsection{Systems of random unconditional convergence in Banach spaces.}

Recall that a sequence $\{x_k\}_{k=1}^\infty$ of elements of a Banach space X
is called {\it basic} if it is a basis in its closed linear span. A sequence $\{x_k\}_{k=1}^\infty$ is said to be an unconditional basic sequence if the sequence $\{x_{\pi(k)}\}_{k=1}^\infty$ is basic in $X$ for any permutation $\pi$ of the set of positive integers. It is well known that a basic sequence $\{x_k\}_{k=1}^\infty$  is unconditional in the Banach space $X$ if and only if multiplication by signs does not affect (up to a constant) the norm of a linear combination of elements of this sequence, that is, there exists a constant $D > 0$ such that, for every $n\in\mathbb{N}$, any signs $\theta_k=\pm1$ and all $a_k\in\mathbb{R}$, it holds 
$$
\Big\|\sum_{k=1}^n\theta_ka_kx_k\Big\|_X\leqslant D\Big\|\sum_{k=1}^na_kx_k\Big\|_X.
$$

It is clear that the Rademacher system is an unconditional sequence in every rearrangement invariant space.  More precisely,  for each $n\in\mathbb{N}$ and all $\theta_{k}=\pm 1$ the sequences $\{r_k\}_{k=1}^n$ and $\{\theta_{k}r_k\}_{k=1}^n$ are identically distributed on $[0,1]$. Therefore, if $X$ is an rearrangement invariant space, then for each $n\in\mathbb{N}$ and all $a_k\in\mathbb{R}$, $k=1,\dots,n$, we have 
$$
\Big\|\sum\nolimits_{k=1}^n\theta_ka_kr_k\Big\|_X=\Big\|\sum\nolimits_{k=1}^na_kr_k\Big\|_X.
$$

A detailed account of properties of basic and unconditional basic sequences can be found, for instance, in the books \cite{AK,KashinSaakyan,Bra}.

The next notion is a natural relaxation of the previous one.
%определения безусловной базисной последовательности. 

\begin{defin} 
%(см. \cite{LAT}) 
A basic sequence $\{x_k\}_{k=1}^\infty$ in a Banach space $X$ is called a {\it system of random unconditional convergence} (in brief, a RUC system) whenever there exists a constant $D>0$ such that for any $n\in\mathbb{N}$ and $a_k\in\mathbb{R}$, $k=1,2,\ldots,n,$ we have
$$
\mathsf{E}_\theta\Big\|\sum_{k=1}^n\theta_ka_kx_k\Big\|_X\leqslant D\Big\|\sum_{k=1}^na_kx_k\Big\|_X,
$$
where the expectation is taken over all $\theta_k=\pm 1$, $k=1,2,\dots,n$.
%{\it RUC} системой будем называть $D$-RUC систему с некоторой константой $D$, точное значение которой для нас не важно.
\end{defin}

%Обозначение RUC является аббревиатурой выражения "Random Unconditional Convergence". Понятие RUC системы было введено в работе \cite{BKPS}; там же были доказаны многие ва
This concept was introduced by Billard, Kwapie\'n, Pe{\l}czy\'nski, and Samuel in \cite{BKPS}. This paper contains also a number of interesting results related to the RUC-property of systems. Subsequently, this research was continued by many authors studied the behaviour of RUC systems in various function spaces (see e.g. \cite{wojtaszczyk86,garling-tomczak-jaegermann,dodds-semenov-sukochev, LAT, astashkin-curbera-tikhomirov,astashkin-curbera}).

\begin{remark} 
One can easily check that for each $n\in\mathbb{N}$, all $\eta_{k}=\pm 1$ and $a_k\in\mathbb{R}$, $k=1,\dots,n$,
$$
\mathsf{E}_\theta\Big\|\sum_{k=1}^n \theta_{k}\eta_ka_kx_k\Big\|_X=
\mathsf{E}_\theta\Big\|\sum_{k=1}^n \theta_ka_kx_k\Big\|_X.
$$
 Thus, $\{x_k\}_{k=1}^\infty$ is a RUC system in a Banach space $X$ if and only if for some constants independent of $n\in\mathbb{N}$ and $a_k\in\mathbb{R}$, $k=1,\dots,n$, 
$$
\mathsf{E}_\theta\Big\|\sum_{k=1}^n \theta_ka_kx_k\Big\|_X\asymp \min_{\theta_{k}=\pm 1}\Big\|\sum_{k=1}^n\theta_{k}a_kx_k\Big\|_X.
$$ 
\end{remark}

\subsection{Multiple Rademacher system, Rademacher chaos and decoupling.}
\label{decoupling}    
$\;$
%{\color{red} Let $d\in\mathbb{N}$. For every multi-index $\vec{j}=(j_1,j_2,\ldots,j_d)\in  \mathbb{N}^d$ we set 
%$$\mathbf{r}_{\vec{j}}^\otimes (t_1,\dots,t_d):=(r_{j_1}\otimes\ldots\otimes r_{j_d})(t_1,\dots,t_d):=r_{j_1}(t_1)\cdot\ldots\cdot r_{j_d}(t_d),\;\;(t_1,\dots,t_d)\in [0,1]^d.$$}
Let $d\in\mathbb{N}$ and let
$$
\mathbb{N}^d:=\{{\vec{j}}=(j_1,j_{2},\ldots,j_d):\, j_k\in\mathbb{N},\, k=1,2,\ldots,d \}.
$$
Further, for each ${\vec{j}}=(j_1,j_{2},\ldots,j_d)\in\mathbb{N}^d$ we set
$$
{\rm r}_{{\vec{j}}}^{\otimes}(\vec{t})=r_{j_1}(t_1)\cdot r_{j_2}(t_2)\cdot\dots\cdot r_{j_d}(t_d),\;\;\mbox{where}\;\vec{t}=(t_1,t_2,\dots,t_d)\in [0,1]^d.$$
Clearly, the sequence $\{\mathbf{r}_{\vec{j}}^\otimes\}_{\vec{j}\in \mathbb{N}^d}$ (which is referred as {\it the multiple Rademacher system of order $d$}) is  
uniformly bounded and orthonormal on the cube $[0,1]^d$. Moreover, ordered in a certain natural way this system is basic in any rearrangement invariant space on $[0,1]$ \cite[Theorem~2]{AL_AA}, and it satisfies for every $1\le p<\infty$, by the well-known Bonami result (see \cite{Bonami70} or \cite[Chapter VII, Exercize 32]{Blei}), the following Khintchine's type estimates:
\begin{equation}
\label{Bonami}
A_{p,d}\|(a_{\vec{\jmath}})_{{\vec{\jmath}}\in \mathbb{N}^d}\|_{{\ell^2}}
\leqslant \Big\|\sum_{\vec{\jmath}\in \mathbb{N}^d}a_{\vec{\jmath}}r_{\vec{\jmath}}^\otimes\Big\|_{L_p([0,1]^d)}\leqslant B_{p,d}\|(a_{\vec{\jmath}})_{{\vec{\jmath}}\in \mathbb{N}^d}\|_{{\ell^2}},
\end{equation}
where the constants $A_{p,d}$ and $B_{p,d}$ depend only on $d$ and $p$. 

In the more general setting, it is known (see \cite[Theorem 6]{AL_AA} and  earlier versions of this result in \cite{As2000}) that, for an rearrangement invariant space $X$,  the following three conditions are equivalent: 

(i) the sequence $\{\mathbf{r}_{\vec{j}}^\otimes\}_{\vec{j}\in \mathbb{N}^d}$ is equivalent in $X$ to the canonical basis in $\ell_2$; 

(ii) $\{\mathbf{r}_{\vec{j}}^\otimes\}_{\vec{j}\in \mathbb{N}^d}$ is an unconditional sequence in $X$; 

(iii) $X\supset G_{2/d}$. 

Here, $G_{r}$ is the separable part of the exponential Orlicz space 
$\mathrm{Exp}L^{r}$ generated by an Orlicz function $N_r(u)$ equivalent to the function $e^{u^{r}}$ for large $u>0$ (see footnote in Subsect. \ref{Sec_Khintchine}).
%(see \cite{AL_AA-24} and  earlier versions of this result in \cite{As2000} and \cite{AL_AA}). 
Moreover, each uniformly bounded orthonormal system (and hence the multiple Rademacher system) has the RUC property in an rearrangement invariant space $X$ such that $X\supset \mathrm{Exp}L^2$ \cite[Proposition 2]{AL_AA-24}.

Recall now the definition of the Rademacher chaos.  For each $d\in\mathbb{N}$ we introduce the following notation:
%{\color{red} 
%$$
%\bigtriangleup^d:=\{(j_1,j_2,\ldots,j_d)\in\mathbb{N}^d:\;j_1>j_2>\ldots>j_d\}.$$}
%{\color{blue} 
$$
\Delta^d:=\{(j_1,j_2,\ldots,j_d)\in\mathbb{N}^d:\;j_1<j_2<\ldots<j_d\}.$$
 The {\it (homogeneous) Rademacher chaos of order $d$} consists of all functions of the form
$$\mathbf{r}_{{{{\vec{j}}}}}(t):=r_{j_1}(t)\cdot r_{j_2}(t)\cdot\ldots\cdot r_{j_d}(t),\;\;t\in [0,1],$$ where ${{{\vec{j}}}}=(j_1,j_2,\ldots,j_d)\in \Delta^d$. As the multiple Rademacher system, the system $\{\mathbf{r}_{\vec{j}}\}_{\vec{j}\in \Delta^d}$ when naturally ordered (see \cite[Theorem~2]{AL_AA}) is basic in any rearrangement invariant space on $[0,1]$ and satisfies two-sided $L_p$-estimates for $1\le p<\infty$ similar to \eqref{Bonami}.
%Then, $\{\mathbf{r}_{\vec{j}}\}_{{{{\vec{j}}}}\in \bigtriangleup^d}$ is a RUC sequence in $L_\infty(I)$.

The study of properties of the Rademacher chaos can be reduced often to that of similar ones for the multiple Rademacher system by applying the decoupling techniques. Further, we will repeatedly use the following result of such a sort, where for all $d,n\in\mathbb{N}$ we set
$$
\mathbb{N}_n^d:=\{(i_1,i_2,\ldots,i_d):\,i_k\in\{1,2,\ldots,n\},\,k=1,2,\ldots, d\},$$
and, if $d\le n$,
$$
\widetilde{\mathbb{N}_n^d}:=\{(i_1,i_2,\ldots,i_d)\in \mathbb{N}_n^d:\,i_p\neq i_q\;\mbox{if}\;p\neq q\}
$$ 
 (see \cite[Theorem 3.1.1]{DG}).

\begin{theorem}
\label{real theor_decoupling, general} 
Let $(\xi_1,\xi_2,\ldots,\xi_n)$ be a sequence of $n$ real-valued independent random variables and let $(\xi_1^k,\xi_2^k,\ldots,\xi_n^k)$, $k=1,2,\dots,d$, be $d$ independent copies of this sequence, $d\le n$. Assume also that, for each  $(i_1,\dots,i_d)\in \widetilde{\mathbb{N}_n^d}$, measurable functions $h_{i_1,\dots,i_d}:\,\mathbb{R}^d\to \mathbb{R}$ satisfy the condition: $\mathsf{E}(|h_{i_1,\dots,i_d}(\xi_{i_1},\xi_{i_2},\ldots,\xi_{i_d})|)<\infty$. Let $\Phi:\,[0,\infty)\to [0,\infty)$ be a convex nondecreasing function such that 
$$\mathsf{E}\Phi(|h_{i_1,\dots,i_d}(\xi_{i_1},\xi_{i_2},\ldots,\xi_{i_d})|)<\infty$$ for each $(i_1,\dots,i_d)\in \widetilde{\mathbb{N}_n^d}$. Then,
\begin{equation}
\label{main decoupling}
\mathsf{E}\Phi\Big(\Big|\sum_{(i_1,\dots,i_d)\in\widetilde{\mathbb{N}_n^d}}h_{i_1,\dots,i_d}(\xi_{i_1},\ldots,\xi_{i_d})\Big|\Big)\le \mathsf{E}\Phi\Big(C_d(1)\Big|\sum_{(i_1,\dots,i_d)\in\widetilde{\mathbb{N}_n^d}}h_{i_1,\dots,i_d}(\xi_{i_1}^1,\ldots,\xi_{i_d}^d)\Big|\Big),
\end{equation}
where $C_d(1):=2^d(d^d-1)((d-1)^{(d-1)}-1)\times\dots\times 3$.

If, moreover, the functions $h_{i_1,\dots,i_d}$ are symmetric in the sense that  for an arbitrary permutation $\pi$ of the set $\{1,2,\dots,d\}$, each $(i_1,\dots,i_d)\in \widetilde{\mathbb{N}_n^d}$ and all $t_i\in\mathbb{R}$, $i=1,\dots,d$, we have 
$$
h_{i_1,\dots,i_d}(t_1,\dots,t_d)=h_{i_{\pi(1)},\dots,i_{\pi(d)}}(t_{\pi(1)},\dots,t_{\pi(d)}),$$
then inequality \eqref{main decoupling} can be reversed:
$$
\mathsf{E}\Phi\left(C_d(2)^{-1}\Big|\sum_{(i_1,\dots,i_d)\in\widetilde{\mathbb{N}_n^d}}h_{i_1,\dots,i_d}(\xi_{i_1}^1,\ldots,\xi_{i_d}^d)\Big|\right)\le \mathsf{E}\Phi\Big(\Big|\sum_{(i_1,\dots,i_d)\in\widetilde{\mathbb{N}_n^d}}h_{i_1,\dots,i_d}(\xi_{i_1},\ldots,\xi_{i_d})\Big|\Big),$$
where $C_d(2):=2^{(2d-2)}(d-1)!$.
\end{theorem}

In particular, applying this theorem in the special case when $\Phi(t)=t^p$, $1\le p<\infty$, and $h_{i_1,\dots,i_d}$ are multilinear functions, we obtain the following result (see also \cite[Theorem 6.4.1]{KwW}). %Further, we will use the following result of such a sort which is an   immediate consequence of Theorem 3.1.1 from the book \cite{DG}.

\begin{cor}
\label{real theor_decoupling} 
Let $d,n\in\mathbb{N}$, $d\le n$, and let sequences of real-valued random variables $(\xi_1,\xi_2,\ldots,\xi_n)$ and $(\xi_1^k,\xi_2^k,\ldots,\xi_n^k)$, $k=1,2,\dots,d$, defined on probability spaces $\Omega$ and $\Omega_k$, respectively, satisfy the conditions of Theorem \ref{real theor_decoupling, general}. Suppose that for every permutation $\pi$ of the set $\{1,2,\dots,d\}$ and each $(j_1,\dots,j_d)\in \widetilde{\mathbb{N}_n^d}$, the following holds: $d_{j_1,\dots,j_d}=d_{j_{\pi(1)},\dots,j_{\pi(d)}}$.

Then, we have
%{\color{red}
%$$
%C_d(2)^{-1}\Big\|\sum_{(i_1,\dots,i_d)\in\widetilde{\mathbb{N}_n^d}}d_{i_1,\dots,i_d}\xi_{i_1}^1\ldots\xi_{i_d}^d\Big\|_{L_\infty([0,1]^d)}\le \Big\|\sum_{(i_1,\dots,i_d)\in\widetilde{\mathbb{N}_n^d}}d_{i_1,\dots,i_d}\xi_{i_1}\ldots\xi_{i_d})\Big\|_{L_\infty([0,1])}
%$$
%$$
%\le C_d(1)\Big\|\sum_{(i_1,\dots,i_d)\in\widetilde{\mathbb{N}_n^d}}d_{i_1,\dots,i_d}\xi_{i_1}^1\ldots\xi_{i_d}^d\Big\|_{L_\infty([0,1]^d)},
%$$}
%{\color{blue}
$$
C_d(2)^{-1}\Big\|\sum_{(i_1,\dots,i_d)\in\widetilde{\mathbb{N}_n^d}}d_{i_1,\dots,i_d}\xi_{i_1}^1\ldots\xi_{i_d}^d\Big\|_{L_\infty(\prod_{k=1}^d\Omega_k)}\le \Big\|\sum_{(i_1,\dots,i_d)\in\widetilde{\mathbb{N}_n^d}}d_{i_1,\dots,i_d}\xi_{i_1}\ldots\xi_{i_d})\Big\|_{L_\infty(\Omega)}
$$
$$
\le C_d(1)\Big\|\sum_{(i_1,\dots,i_d)\in\widetilde{\mathbb{N}_n^d}}d_{i_1,\dots,i_d}\xi_{i_1}^1\ldots\xi_{i_d}^d\Big\|_{L_\infty(\prod_{k=1}^d\Omega_k)},
$$
where $C_d(1)$ and $C_d(2)$ are the constants from Theorem \ref{real theor_decoupling, general}. 

\end{cor}
\begin{proof}

Observe that the functions
$$
h_{j_1,j_2,\dots,j_d}(t_1,t_2,\dots,t_d):=d_{j_1,j_2\dots,j_d}\cdot t_1\cdot t_2\cdot\ldots\cdot t_d$$
are symmetric in the sense indicated in the second part of Theorem \ref{real theor_decoupling, general}. Indeed, by the assumption, for every permutation $\pi$ of the set $\{1,2,\dots,d\}$ and each $(i_1,\dots,i_d)\in \widetilde{\mathbb{N}_n^d}$, we have
\begin{eqnarray*}
h_{j_{\pi(1)},j_{\pi(2)},\dots,j_{\pi(d)}}(t_{\pi(1)},t_{\pi(2)},\dots,t_{\pi(d)})&=&
d_{j_{\pi(1)},j_{\pi(2)}\dots,j_{\pi(d)}}\cdot t_{\pi(1)}\cdot t_{\pi(2)}\cdot\ldots\cdot t_{\pi(d)}\\
&=&
d_{j_{1},j_{2}\dots,j_{d}}\cdot t_1\cdot t_2\cdot\ldots\cdot t_d\\
&=& h_{j_1,j_2,\dots,j_d}(t_1,t_2,\dots,t_d).
\end{eqnarray*}

Thus, from Theorem \ref{real theor_decoupling, general} it follows that, for every $p\ge 1$,
\begin{eqnarray*}
C_d(2)^{-p}\mathsf{E}\Big(\Big|\sum_{(i_1,\dots,i_d)\in\widetilde{\mathbb{N}_n^d}}d_{i_1,\dots,i_d}\xi_{i_1}^1\ldots\xi_{i_d}^d\Big|\Big)^p&\le& \mathsf{E}\Big(\Big|\sum_{(i_1,\dots,i_d)\in\widetilde{\mathbb{N}_n^d}}d_{i_1,\dots,i_d}\xi_{i_1}\ldots\xi_{i_d})\Big|\Big)^p\\
&\le& C_d(1)^p\mathsf{E}\Big(\Big|\sum_{(i_1,\dots,i_d)\in\widetilde{\mathbb{N}_n^d}}d_{i_1,\dots,i_d}\xi_{i_1}^1\ldots\xi_{i_d}^d\Big|\Big)^p.
\end{eqnarray*}
Raising now all terms of this inequality to the power $1/p$ and passing then to the limit as $p\to\infty$, we arrive at the desired estimate. This completes the proof.

\end{proof}

%\begin{cor}
%\label{real theor_decoupling} 
%Suppose a random vector $\xi=(\xi_1,\xi_2,\ldots,\xi_n)$ consists of independent and bounded components $\xi_i$, $i\in\overline{[1,n]}$, and $\eta=(\eta_1,\eta_2,\ldots \eta_n)$ is a vector that is identically distributed with $\xi$ and is independent of $\xi$. Then for every $p>1$ and any real symmetric matrix $B=(b_{i,j})$ we have
%$$
%\frac{1}{12^{p}}\mathsf{E}\Bigl|\sum_{i\neq j} b_{i,j}\xi_i\xi_j\Bigr|^p\leqslant \mathsf{E}\Bigl|\sum_{i\neq j} b_{i,j}\xi_i\eta_j\Bigr|^p\leqslant 4^p \mathsf{E}\Bigl|\sum_{i\neq j} b_{i,j}\xi_i\xi_j\Bigr|^p.
%$$
%Hence,
%$$
%\frac{1}{12}\Bigl\|\sum_{i\neq j} b_{i,j}\xi_i\xi_j\Bigr\|_\infty\leqslant \Bigl\|\sum_{i\neq j} b_{i,j}\xi_i\eta_j\Bigr\|_\infty\leqslant 4 \Bigl\|\sum_{i\neq j} b_{i,j}\xi_i\xi_j\Bigr\|_\infty.
%$$

%\end{cor}

\subsection{Connections between $L_\infty$-norms of Rademacher sums and special matrix norms}
\label{Sec_matrix_norm}

It is crucial for what follows that $L_\infty$-norms of linear combinations of  elements of both sequences $\{\mathbf{r}_{\vec{j}}^\otimes\}_{\vec{j}\in \mathbb{N}^d}$ and $\{\mathbf{r}_{{{{\vec{j}}}}}\}_{\vec{j}\in \Delta^d}$ are determined (up to equivalence) by some special norms of matrices of their coefficients. This fact refines some classical results, in particular, the well-known Littlewood  $4/3$-inequality \cite{Litt1}, giving new information on geometric properties of the subspaces of $L_\infty$ spanned by the above systems. 

We start with the second-order multiple Rademacher system.
%$$
%\sum_{i=1}^n\sum_{j=1}^m a_{i,j}r_i(u)r_j(v)
%$$
%в терминах специальных норм матрицы
%$$
%A=(a_{i,j})_{i\leqslant n, j \leqslant m}.
%$$
Since the set of extreme points of the unit ball of the space $\ell_\infty^n$ consists of the vectors $(\epsilon_1,\epsilon_2,\ldots,\epsilon_n)$, $\epsilon_k=\pm1$, $k=1,2,\dots,n$, we obtain
\begin{eqnarray}
\Bigl\|\sum_{i=1}^n\sum_{j=1}^m a_{i,j}r_i\otimes r_j\Bigr\|_{L_\infty([0,1]^2)}&=&\max_{\epsilon_i, \epsilon'_j=\pm 1} \Bigl | \sum_{i=1}^n\sum_{j=1}^m a_{i,j}\epsilon_i\epsilon'_j\Bigr|\nonumber\\
&=& \sup_{{\|x\|}_{\ell_\infty^n}={\|y\|}_{\ell_\infty^m}=1}\Bigl \{  \sum_{i=1}^n\sum_{j=1}^m a_{i,j}x_iy_j\Bigr\},
\label{bilinear}
\end{eqnarray}
where $x=(x_i)_{i=1}^n$, $y=(y_j)_{j=1}^m$.
Consequently, from the fact that $(\ell_1^n)^*=\ell_\infty^n$ it follows 
%Следовательно, так как сопряженным к пространству $\ell_\infty^n$ является пространство $\ell_1^n$,
\begin{equation}\label{eq_operator_norm}
\Bigl\|\sum_{i=1}^n\sum_{j=1}^m a_{i,j}r_i\otimes r_j\Bigr\|_{L_\infty([0,1]^2)}=\|A:\,\ell_\infty^m\to \ell_1^n\|,
\end{equation}
where the expression in the right-hand side denotes the norm of the operator $A:\,\ell_\infty^m\to \ell_1^n$, generated by the matrix $A=(a_{i,j})_{i\leqslant n, j \leqslant m}$.

An important role in the study of the efficiency of approximative and computational algorithms is played by the so-called cut-norm of a matrix $A=(a_{i,j})_{i\leqslant n,j\leqslant m}$ defined as follows:
\begin{equation}\label{def_cut_norm}
{\|A\|}_{cut}:=\max\Big\{\Big|\sum_{i\in I,j\in J}a_{i,j}\Big|:\,I\subset[n], J\subset [m]\Big\}
\end{equation}
(see e.g. \cite{AVKK} and its references, \cite{Alon-Naor}, \cite{Alon-15}).
As was shown by Alon and Naor \cite[Lemma 3.1]{Alon-Naor}, for every $n,m\in\mathbb{N}$ and any matrix $A=(a_{i,j})_{i\le n,j\le m}$, one has
\begin{equation}\label{cut_norm_matrix}
{\|A\|}_{cut}\leqslant \|A:\,\ell_\infty^m\to \ell_1^n\|\leqslant 4{\|A\|}_{cut},
\end{equation}
and hence, by \eqref{eq_operator_norm}, it follows that
\begin{equation}\label{eq_rad_cut_norm}
{\|A\|}_{cut}\le \Bigl\|\sum_{i=1}^n\sum_{j=1}^m a_{i,j}r_i\otimes r_j\Bigr\|_{L_\infty([0,1]^2)}\le 4 {\|A\|}_{cut}.
\end{equation}
%с константами, не зависящими от $m,n$ и $\{a_{i,j}\}$.

In the case of the Rademacher chaos, we have to deal with $L_\infty$-norms of the sums
\begin{equation}\label{first def}
S_n:=\sum_{i=1}^n\sum_{j=i+1}^n a_{i,j}r_ir_j,\;\;n\in\mathbb{N}.
\end{equation}
Observe that 
$$
S_n=\sum_{i=1}^n\sum_{j=1}^n d_{i,j}r_ir_j,
$$
where $d_{i,j}=d_{j,i}=a_{i,j}/2$ if $j>i$, and $d_{i,i}=0$. Then, applying Corollary \ref{real theor_decoupling} in the case $d=2$ together with equivalence \eqref{eq_rad_cut_norm}, we obtain
\begin{equation}
\label{first appl dec}
\|S_n\|_{L_\infty([0,1])}\asymp \Bigl\|\sum_{i=1}^n\sum_{j=1}^n d_{i,j}r_i\otimes r_j\Bigr\|_{L_\infty([0,1]^2)}\asymp {\|D\|}_{cut},
\end{equation}
where $D=(d_{i,j})_{1\leqslant i,j\leqslant n}$. 

Now, we intend to replace the norm ${\|D\|}_{cut}$ in \eqref{first appl dec} with some modified version of the cut-norm of the matrix $A$. To this end, consider  Bernoulli random variables $b_{i}(t):=(r_i(t)+1)/2$, $t\in [0,1]$. Then, for each $(s,t)\in [0,1]\times[0,1]$ we have
$$
\sum_{i=1}^n\sum_{j=1}^n d_{i,j}b_i(s)b_j(t)=\sum_{i\in I(s)}\sum_{j\in J(t)}d_{i,j},$$
with $I(s):=\{i:\,b_i(s)=1\}$ and $J(t):=\{j:\,b_j(t)=1\}$.
This relation combined with the fact that, for each arrangement of signs $(\eta_1,\eta_2,\dots,\eta_n)$ there is $t\in [0,1]$ such that $r_k(t)=\eta_k$ for all $k=1,\dots,n$, implies the equality 
\begin{equation}
\label{just equality}
{\|D\|}_{cut}=\Bigl\|\sum_{i=1}^n\sum_{j=1}^n d_{i,j}b_i\otimes b_j\Bigr\|_{L_\infty([0,1]^2)}.
\end{equation}

On the other hand, since for each $s\in [0,1]$
$$
\sum_{i=1}^n\sum_{j=i+1}^n a_{i,j}b_i(s)b_j(s)=\sum_{i,j\in I(s),\, i<j}a_{i,j},$$
%where $I(s):=\{i:\,b_i(t)=1\}$. Hence,
we have
\begin{equation}
\label{second equality}
\Bigl\|\sum_{i=1}^n\sum_{j=i+1}^n a_{i,j}b_ib_j \Bigr\|_{L_\infty}={\|A\|}_{cut}^*,
\end{equation}
where we put
\begin{equation}\label{def_square}
{\|A\|}_{cut}^*:=\max\Big\{\Big|\sum_{i,j\in I,\, i<j}a_{i,j}\Big|:\,I\subset [n]\Big\}.
\end{equation}

Next, one can easily see that
$$
\sum_{i=1}^n\sum_{j=i+1}^n a_{i,j}b_ib_j=\sum_{i=1}^n\sum_{j=1}^n d_{i,j}b_ib_j.$$
Hence, applying Corollary \ref{real theor_decoupling} once more, but now to the  random variables $b_{i}$, $i=1,2,\dots,n$, we obtain
$$
\Bigl\|\sum_{i=1}^n\sum_{j=i+1}^n a_{i,j}b_ib_j \Bigr\|_{L_\infty}
%=\Bigl\|\sum_{i=1}^n\sum_{j=i}^n d_{i,j}b_ib_j\Bigr\|_{L_\infty([0,1])}
\asymp \Bigl\|\sum_{i=1}^n\sum_{j=1}^n d_{i,j}b_i\otimes b_j \Bigr\|_{L_\infty([0,1]^2)}.$$
%$$
%\Bigl\|\sum_{i=1}^n\sum_{j=1}^n d_{i,j}b_i(u)b_j(v)\Bigr\|_{L_\infty([0,1]^2)}\asymp \Bigl\|\sum_{i=1}^n\sum_{j=i+1}^n a_{i,j}b_i(t)b_j(t)\Bigr\|_{L_\infty([0,1])}={\|A\|}_{\square},
%$$
Combining this together with \eqref{just equality} and \eqref{second equality}, we conclude that ${\|D\|}_{cut}\asymp {\|A\|}_{cut}^*$ with universal constants.
Therefore, by \eqref{first appl dec} and \eqref{first def}, we arrive at the equivalence
\begin{equation}\label{eq_square_cut_norm}
\Bigl\|\sum_{i=1}^n\sum_{j=i+1}^n a_{i,j}r_ir_j\Bigr\|_{L_\infty}\asymp {\|A\|}_{cut}^*,
\end{equation}
with constants independent of $n\in\mathbb{N}$ and $A=(a_{i,j})_{i\le n,j\le n}$. 

Later on, we will use equivalences \eqref{eq_rad_cut_norm} and  \eqref{eq_square_cut_norm} to obtain sharp (up to universal constants) estimates  for the discrepancy of edge-weighted graphs. To prove similar results for hypergraphs we will need analogous equivalences for the multiple Rademacher system and the Rademacher chaos of an arbitrary order.

We show first that for every $d\in\mathbb{N}$, $n_1,n_2,\ldots n_d\in\mathbb{N}$ and each $d$-dimensional array $A=(a_{i_1,i_2,\ldots , i_d})_{i_k\leqslant n_k,\,1\le k\le d}$ the following inequalities hold:
\begin{equation}\label{eq_rad_cut_norm_mult}
{\|A\|}_{cut}\leqslant\Bigl\|\sum_{i_1=1}^{n_1}\sum_{i_2=1}^{n_2}\ldots \sum_{i_d=1}^{n_d} a_{i_1,i_2\ldots, i_d} r_{i_1}\otimes r_{i_2}\otimes\ldots r_{i_d}\Bigr\|_{L_\infty([0,1]^d)}\leqslant 2^d{\|A\|}_{cut},
\end{equation}
where the cut-norm ${\|A\|}_{cut}$ is defined by
\begin{equation}\label{eq_rad_cut_norm_mult def}
{\|A\|}_{cut}:=\max\Bigl\{\Bigl|\sum_{i_1\in I_1}\sum_{i_2\in I_2}\ldots \sum_{i_d\in I_d} a_{i_1,i_2,\ldots, i_d}\Bigr |:\;I_k\subset [n_k],\, k=1,2,\ldots, d\Bigr\}.
\end{equation}

Denoting
$$
S_{n_1,n_2,\ldots n_d}(u_1,u_2,\dots,u_d):=\sum_{i_1=1}^{n_1}\sum_{i_2=1}^{n_2}\ldots \sum_{i_d=1}^{n_d} a_{i_1,i_2\ldots, i_d} r_{i_1}\otimes r_{i_2}\otimes\ldots r_{i_d}(u_1,u_2,\dots,u_d),$$
for any fixed $u_k\in [0,1]$, $k=1,2,\dots,d$, we have
%{\color{red}$$
%S_{n_1,\ldots n_d}(u_1,\dots,u_d):=\sum_{i_1\in I_1(u_1)}\sum_{i_2\in I_2(u_2)}\ldots \sum_{i_d\in I_d(u_d)} a_{i_1,i_2,\ldots, i_d}\epsilon_1 \epsilon_2\dots \epsilon_d,$$}
$$
S_{n_1,\ldots n_d}(u_1,\dots,u_d)=\sum_{\epsilon_k\in\{-1,1\}}\epsilon_1\cdot\dots \cdot\epsilon_d\sum_{i_1\in I_1(u_1,\epsilon_1)}\sum_{i_2\in I_2(u_2,\epsilon_1)}\ldots \sum_{i_d\in I_d(u_d,\epsilon_d)} a_{i_1,i_2,\ldots, i_d},$$
where 
%{\color{red}$\epsilon_k=\pm 1$ and $I_k(u_k):=\{i_k\in [n_k]:\,r_{i_k}(u_k)=\epsilon_k\}$} {\color{blue}  
$I_k(u_k,\epsilon_k):=\{i_k\in [n_k]:\,r_{i_k}(u_k)=\epsilon_k\}$, $k=1,2,\dots,d$.
From this formula and definition \eqref{eq_rad_cut_norm_mult def} of the norm ${\|A\|}_{cut}$ it follows immediately that
$$
|S_{n_1,n_2,\ldots n_d}(u_1,u_2,\dots,u_d)|\le 2^d {\|A\|}_{cut}$$
for all $u_k\in [0,1]$, $k=1,2,\dots,d$, which implies the right-hand side inequality in \eqref{eq_rad_cut_norm_mult}.

To prove the left-hand side inequality, observe that, similarly as in the bilinear case (see equality \eqref{bilinear}), we have
%{\color{red}
%$$
%\|S_{n_1,n_2,\ldots n_d}\|_{L_\infty([0,1]^d)}=\left\{\sum_{i_1=1}^{n_1}\sum_{i_2=1}^{n_2}\ldots \sum_{i_d=1}^{n_d} a_{i_1,i_2\ldots, i_d} x_{i_1}x_{i_2}\ldots x_{i_d}:\,x_{i_k}\in [-1,1]\right\}
%$$}
%{\color{blue}
$$
\|S_{n_1,n_2,\ldots n_d}\|_{L_\infty([0,1]^d)}=\max\left\{\sum_{i_1=1}^{n_1}\sum_{i_2=1}^{n_2}\ldots \sum_{i_d=1}^{n_d} a_{i_1,i_2\ldots, i_d} x_{i_1}x_{i_2}\ldots x_{i_d}:\,x_{i_k}\in [-1,1]\right\}
$$
(cf. \cite[Theorem I.14]{Blei}\footnote{This equality is referred in the book \cite{Blei} as the multilinear Fr\'{e}chet theorem.}). Therefore, for arbitrary sets $I_k\subset [n_k]$, $k=1,2,\ldots, d$, letting $x_{i_k}=1$ if $i_k\in I_k$ and $x_{i_k}=0$ if $i_k\not\in I_k$, we get
$$
\|S_{n_1,n_2,\ldots n_d}\|_{L_\infty([0,1]^d)}\geq\Bigl|\sum_{i_1\in I_1}\sum_{i_2\in I_2}\ldots \sum_{i_d\in I_d} a_{i_1,i_2,\ldots, i_d}\Bigr |.
$$
%{\color{blue} setting $x_{i_k}=1$ if $i_k\in I_k$, and $x_{i_k}=-1$ if $i_k\not\in I_k$, we get
%$$
%\Bigl|\sum_{i_1\in I_1}\sum_{i_2\in I_2}\ldots \sum_{i_d\in I_d} a_{i_1,i_2,\ldots, i_d}\Bigr |=\Bigl|\sum_{i_1\in I_1}\sum_{i_2\in I_2}\ldots \sum_{i_d\in I_d} a_{i_1,i_2,\ldots, i_d}\frac{x_1+1}{2}\cdot\frac{x_2+1}{2}\cdot\ldots\cdot\frac{x_d+1}{2} \Bigr |
%$$
%$$
%\leq 2^{-d}\sum_{B\in 2^{[d]}}\Bigl|\sum_{i_1\in I_1}\sum_{i_2\in I_2}\ldots \sum_{i_d\in I_d} a_{i_1,i_2,\ldots, i_d} x_{i_1}^{[1\in B]}x_{i_2}^{[2\in B]}\ldots x_{i_d}^{[d\in B]} \Bigr |\leq \|S_{n_1,n_2,\ldots n_d}\|_{L_\infty([0,1]^d)},
%$$
%where $[i\in B]=1$ if $i \in B$ and $[i\in B]=0$ otherwise (Iverson bracket).}
As a result, taking the supremum over all sets $I_k\subset [n_k]$, $k=1,2,\ldots, d$, we arrive at the left-hand side inequality in \eqref{eq_rad_cut_norm_mult}.

To deal with the Rademacher chaos, for every $d,n\in\mathbb{N}$, $d\le n$, and each array $A=(a_{i_1,i_2\ldots, i_d})_{1\le i_1<i_2<\dots< i_d\le n }$ we introduce the following modification of the cut-norm:
\begin{equation}\label{quadrat_mult}
{\|A\|}_{cut}^*:=\max\Bigl\{\Bigl|\sum_{i_1\in I}\sum_{i_2\in I,\,i_2>i_1}\ldots \sum_{i_d\in I,\,i_d>i_{d-1}} a_{i_1,i_2\ldots, i_d}\Bigr |:\;I\subset [n]\Bigr\}.
\end{equation}
Then,
 \begin{equation}\label{eq_square_cut_norm_mult}
\Bigl\|\sum_{i_1=1}^{n}\sum_{i_2=i_1+1}^{n}\ldots \sum_{i_d=i_{d-1}+1}^{n} a_{i_1,i_2\ldots, i_d} r_{i_1}r_{i_2}\ldots r_{i_d}\Bigr\|_{L_\infty}\asymp {\|A\|}_{cut}^*,
\end{equation}
with constants independent of $n\in\mathbb{N}$ and $A$.
% where
%\begin{equation}\label{quadrat_mult}
%{\|A\|}_{cut}^*:=\max\Bigl\{\Bigl|\sum_{i_1\in I}\sum_{i_2\in I,\,i_2>i_1}\ldots \sum_{i_d\in I,\,i_d>i_{d-1}} a_{i_1,i_2\ldots, i_d}\Bigr |:\;I\subset [n]\Bigr\}.
%\end{equation}

The proof of \eqref{eq_square_cut_norm_mult} can be carried out by the same scheme as that of relation \eqref{eq_square_cut_norm} for the second-order chaos. Indeed, let $(j_1,\dots,j_d)\in\mathbb{N}_n^d$. If indices $j_1,\dots,j_d$ are pairwise different, we set 
$$
d_{j_1,\dots,j_d}:=\frac1{d!}a_{j_{\sigma_1},\dots,j_{\sigma_d}},$$ where $\sigma_1,\dots,\sigma_d$ is the permutation of the set $[d]$ such that $j_{\sigma_1}<j_{\sigma_2}<\dots <j_{\sigma_d}$. Otherwise, we put $d_{j_1,\dots,j_d}=0$. Then, the coefficients $d_{j_1,\dots,j_d}$ satisfy the symmetry assumption of Corollary \ref{real theor_decoupling}. Therefore, applying this corollary to the Rademacher sequence $\{r_{i}\}_{i=1}^n$ and its independent copies $\{r_{i}(t_k)\}_{i=1}^n$, $k=1,\dots,d$, we obtain
$$
\Bigl\|\sum_{i_1=1}^{n}\ldots \sum_{i_d=1}^{n} d_{i_1,\ldots, i_d} r_{i_1}\ldots r_{i_d}\Bigl\|_{L_\infty}
\asymp\Bigl\|\sum_{i_1=1}^{n}\ldots \sum_{i_d=1}^{n} d_{i_1,\ldots, i_d} r_{i_1}\otimes\ldots \otimes r_{i_d}\Bigr\|_{L_\infty([0,1]^d)}.
$$
%$$
%h_{j_1,j_2,\dots,j_d}(t_1,t_2,\dots,t_d):=d_{j_1,j_2\dots,j_d}r_{j_1}(t_1)\cdot r_{j_2}(t_2)\cdot\ldots\cdot r_{j_d}(t_d)$$
%are symmetric in the sense that, for all $t_k\in [0,1]$, $k=1,\dots,d$, and any permutation $s_1,s_2,\dots,s_d$ of the numbers $1,2,\dots,d$, we have
%\begin{equation}
%\label{symm}
%h_{j_{s_1},j_{s_2},\dots,j_{s_d}}(t_{s_1},t_{s_2},\dots,t_{s_d})=h_{j_1,j_2,\dots,j_d}(t_1,t_2,\dots,t_d).
%\end{equation}
%Indeed, if $j_1,\dots,j_d$ are pairwise different, then
%\begin{eqnarray*}
%h_{j_{s_1},j_{s_2},\dots,j_{s_d}}(t_{s_1},t_{s_2},\dots,t_{s_d}) &=&b_{j_{s_1},j_{s_2},\dots,j_{s_d}}r_{j_{s_1}}(t_{s_1})\cdot r_{j_{s_2}}(t_{s_2})\cdot\ldots\cdot r_{j_{s_d}}(t_{s_d})\\&=&a_{j_{\sigma_1},\dots,j_{\sigma_d}}r_{j_1}(t_1)\cdot r_{j_2}(t_2)\cdot\ldots\cdot r_{j_d}(t_d),
%\end{eqnarray*}
%for the above permutation $\sigma_1,\dots,\sigma_d$, and
%$$
%h_{j_{s_1},j_{s_2},\dots,j_{s_d}}(t_{s_1},t_{s_2},\dots,t_{s_d}) =0,
%$$
%otherwise. Combining this together with the definition of the functions $h_{j_{1},j_{2},\dots,j_{d}}$, we obtain \eqref{symm}. Therefore, we can apply Corollary \ref{real theor_decoupling}. 
Thus, taking into account inequality \eqref{eq_rad_cut_norm_mult} for the array $D=(d_{j_1,\dots,j_d})_{(j_1,\dots,j_d)\in\mathbb{N}_n^d}$ and observing that, by the definition of coefficients $d_{i_1,\dots,i_d}$, it holds
$$
\sum_{i_1=1}^{n}\ldots \sum_{i_d=1}^{n} d_{i_1,i_2,\ldots, i_d} r_{i_1}r_{i_2}\ldots r_{i_d}=\sum_{i_1=1}^{n}\sum_{i_2=i_1+1}^{n}\ldots \sum_{i_d=i_{d-1}+1}^{n} a_{i_1,i_2\ldots, i_d} r_{i_1}r_{i_2}\ldots r_{i_d},$$
we conclude that
\begin{equation}
\label{first appl dec multiple}
%\|S_n\|_{L_\infty([0,1])}\asymp 
\Bigl\|\sum_{i_1=1}^{n}\sum_{i_2=i_1+1}^{n}\ldots \sum_{i_d=i_{d-1}+1}^{n} a_{i_1,i_2\ldots, i_d} r_{i_1}r_{i_2}\ldots r_{i_d}\Bigr\|_{L_\infty}
\asymp {\|D\|}_{cut}.
\end{equation}
%where $D=(d_{j_1,\dots,j_d})_{(j_1,\dots,j_d)\in\mathbb{N}_n^d}$. 

Finally, precisely as in the case of the second-order chaos, by using Bernoulli random variables $b_{i}(t):=(r_i(t)+1)/2$, $t\in [0,1]$, $i=1,\dots,n$, we can prove that ${\|D\|}_{cut}\asymp {\|A\|}_{cut}^*$ with universal constants. Hence, by \eqref{first appl dec multiple}, we obtain \eqref{eq_square_cut_norm_mult}.

%$$
%{\|S_{b,n}\|}_{L_p([0,1]^d)} =\Big\|\sum_{{\vec{j}}\in \mathbb{N}^{d}_n}b_{\vec{j}} {\rm r}_{{\vec{j}}}\Big\|_{L_p([0,1]^d)}\le C_4(d) \Bigl\|\sum_{{\vec{j}}\in \mathbb{N}^{d}_n}b_{\vec{j}}\mathbf{r}_{\vec{j}}\Bigr\|_{p},$$
%где $C_4(d)=2^{2d-2}(d-1)!$. Переходя к пределу при $p\to\infty$, отсюда получим
%$$
%{\|S_{b,n}\|}_{L_\infty([0,1]^d)} \le C_4(d) \Bigl\|\sum_{{\vec{j}}\in \mathbb{N}^{d}_n}b_{\vec{j}}\mathbf{r}_{\vec{j}}\Bigr\|_{\infty}.$$

%$И, наконец, снова используя \cite[теорема 3.1.1]{DG}, получим
%$$
%\Bigl\|\sum_{i_1=1}^{n}\ldots \sum_{i_d=1}^{n} b_{i_1,\ldots, i_d} r_{i_1}(u_1)\ldots r_{i_d}(u_d)\Bigr\|_{L_\infty([0,1]^d)}
%$$
%\begin{equation}\label{eq_decoupling_ineq}\asymp \Bigl\|\sum_{i_1=1}^{n}\ldots \sum_{i_d=1}^{n} b_{i_1,\ldots, i_d} r_{i_1}(t)\ldots r_{i_d}(t)\Bigr\|_{L_\infty([0,1])},
%\end{equation}
%при выполнении следующих условий на коэффициенты $b_{i_1,\ldots, i_d}$:
%$$
%b_{i_1i_2\ldots i_k}=b_{i_{\pi(1)}i_{\pi(2)}\ldots i_{\pi(k)}}
%$$
%для каждой перестановки $\pi$ на множестве $\{1,2,\ldots,d\}$, и, кроме этого, $b_{i_1,i_2\ldots, i_d}=0$ при совпадении хотя бы двух индексов $i_l=i_m$, $1\leqslant l<m\leqslant n$.

\section{RUC-property of the second-order multiple Rademacher system in $L_\infty$}\label{Sec_mult_Rad}

It is instructive to prove, first, the property of random unconditional convergence in $L_\infty$ for the {\it second-order} multiple Rademacher system. In this case the proof will be substantially relied on representation \eqref{eq_operator_norm} of the $L_\infty$-norm of a Rademacher sum $\sum_{i=1}^n\sum_{j=1}^m a_{i,j}r_i\otimes r_j$ as the operator norm from $\ell_\infty^m$  into $\ell_1^n$ of the matrix $A=(a_{i,j})_{i\le n,j\le m}$ 
and also on some results of the very recent paper \cite{APSS} by Adamczak, Prochno, Strzelecka, and Strzelecki. Further, in Section \ref{Sec_mult_d_Rad}, we will extend this result to the multiple Rademacher system of an arbitrary order.

\begin{theorem}\label{theor_bi_Rad}
With universal constants, for all $n,m\in\mathbb{N}$ and $a_{i,j}\in\mathbb{R}$, $1\leqslant i\leqslant n$, $1\leqslant j\leqslant m$, the following two-sided estimates hold:
\begin{eqnarray*}
\mathsf{E}_\theta\,\Bigl\|\sum_{i=1}^n\sum_{j=1}^m \theta_{i,j} a_{i,j}r_i\otimes r_j\Bigr\|_{L_\infty([0,1]^2)}&\asymp&
\min_{\theta_{i,j}=\pm 1}{\Big\|\sum_{i=1}^n\sum_{j=1}^m \theta_{i,j}a_{i,j}r_i\otimes r_j\Big\|}_{L_\infty([0,1]^2)}\\
&\asymp&\max\Big\{\sum_{i=1}^n\Big(\sum_{j=1}^m a_{i,j}^2\Big)^{1/2},\sum_{j=1}^m\Big(\sum_{i=1}^n a_{i,j}^2\Big)^{1/2}\Big\}. 
\end{eqnarray*}

%\begin{equation*}
%\int\limits_0^1 \Bigl\|\sum_{i=1}^n\sum_{j=1}^m r_{i,j}(\omega) a_{i,j}r_i(u)r_j(v)\Bigr\|_{L_\infty([0,1]^2)}\,d\omega\asymp \min_{\theta_{i,j}=\pm 1}{\Big\|\sum_{i=1}^n\sum_{j=1}^m \theta_{i,j}a_{i,j}r_i(u)r_j(v)\Big\|}_{L_\infty([0,1]^2)}
%\end{equation*}
%\begin{equation*}
%\asymp\max\Big\{\sum_{i=1}^n\Big(\sum_{j=1}^m a_{i,j}^2\Big)^{1/2},\sum_{j=1}^m\Big(\sum_{i=1}^n a_{i,j}^2\Big)^{1/2}\Big\}. 
%\end{equation*}
\end{theorem}

In the proof of this theorem we will make use of two lemmas. The first of them is really well known; in particular, Littlewood applied similar argument in \cite{Litt1} to prove his famous $4/3$-inequality. However, for the convenience of the reader  we present its proof below.
%$$
%{\Big\|\sum_{i,j=1}^\infty a_{i,j}r_i \otimes r_j\Big\|}_{L_\infty([0,1]^2)}\geqslant\frac{1}{\sqrt{2}}\|(a_{i,j})\|_{\ell_{4/3}}.
%$$
%Но ради полноты изложения мы приведем полное доказательство леммы.

\begin{lemma}
\label{Le1}
For any $m,n\in\mathbb{N}$ and $a_{i,j}\in\mathbb{R}$, $1\leqslant i\leqslant n$, $1\leqslant j\leqslant m$, we have
$$
\Big\|\sum_{i=1}^n\sum_{j=1}^m a_{i,j} r_{i}\otimes r_j\Big\|_{L_\infty([0,1]^2)}\geqslant \frac{1}{\sqrt{2}}\max\Big\{\sum_{i=1}^n\Big(\sum_{j=1}^m a_{i,j}^2\Big)^{1/2},\sum_{j=1}^m\Big(\sum_{i=1}^n a_{i,j}^2\Big)^{1/2}\Big\}. 
$$
\end{lemma}
\begin{proof}
By Szarek's inequality \eqref{basic1},
%(with the optimal constant in Khintchine's inequality for $p=1$ ), for each $m\in\mathbb{N}$ and an arbitrary sequence $\{c_j\}_{j=1}^m$ 
it follows that
%$$
%\int_0^1\Big|\sum_{j=1}^m c_jr_j(v)\Big|\,dv\geqslant \frac{1}{\sqrt{2}}\Big(\sum_{j=1}^mc_j^2\Big)^{1/2}.
%$$
%Consequently,
\begin{eqnarray*}
\Big\|\sum_{i=1}^n\sum_{j=1}^m a_{i,j} r_{i}\otimes r_j\Big\|_{L_\infty([0,1]^2)} &=& \sup_{v\in [0,1]}\sum_{i=1}^n\Big|\sum_{j=1}^m a_{i,j} r_{j}(v)\Big|\\ &\geqslant& \sum_{i=1}^n\int_0^1\Big|\sum_{j=1}^m a_{i,j} r_{j}(v)\Big|\,dv\\ &\geqslant& \frac{1}{\sqrt{2}}\sum_{i=1}^n\Big(\sum_{j=1}^m a_{i,j}^2\Big)^{1/2}.
\end{eqnarray*}
Since similarly 
$$
\Big\|\sum_{i=1}^n\sum_{j=1}^m a_{i,j} r_{i}\otimes r_j\Big\|_{L_\infty([0,1]^2)}\geqslant \frac{1}{\sqrt{2}}\sum_{j=1}^m\Big(\sum_{i=1}^n a_{i,j}^2\Big)^{1/2}, 
$$
the proof is completed.
\end{proof}

%As was already said, we will make use of results of the paper \cite{APSS}. 
Further, we will need some additional notation. By $A\circ B$ we denote the Hadamard (i.e., coordinate-wise) multiplication of two matrices $A$ and $B$ of the same size. Let $r_{i,j}$ (respectively $g_{i,j}$) be the Rademacher (respectively, independent standard Gaussian) random variables, $1\leqslant i\leqslant n$, $1\leqslant j\leqslant m$. Setting $R:=(r_{i,j})_{i\leqslant n, j\leqslant m}$ and $G:=(g_{i,j})_{i\leqslant n, j\leqslant m}$, for an arbitrary matrix $A=(a_{i,j})_{i\leqslant n, j\leqslant m}$, we denote by $R_A$ (resp. $G_A$) the operator corresponding to the matrix $R\circ A=(r_{i,j}a_{i,j})_{i\leqslant n, j\leqslant m}$ (respectively, $G\circ A=(g_{i,j}a_{i,j})_{i\leqslant n, j\leqslant m}$). 
%In addition, letwe will use the matrix $A\circ A=(a_{i,j}^2)_{i\leqslant n, j\leqslant m}$. 
By $\mathsf{E}$ we will denote the expectation over the corresponding probability space.
% on which random variables $r_{i,j}$ (resp. $g_{i,j}$) are defined.

\begin{lemma}
\label{Le2}
There is a universal constant $C>0$ such that for any $m,n\in\mathbb{N}$ and $a_{i,j}\in\mathbb{R}$, $1\leqslant i\leqslant n$, $1\leqslant j\leqslant m$, we have
$$
\mathsf{E}\,\|R_A:\,\ell_\infty^m\to \ell_1^n\|\leqslant C\max\Big\{\sum_{i=1}^n\Big(\sum_{j=1}^m a_{i,j}^2\Big)^{1/2},\sum_{j=1}^m\Big(\sum_{i=1}^n a_{i,j}^2\Big)^{1/2}\Big\}.
$$

\end{lemma}
\begin{proof}
In view of \cite[Proposition~1.8(i)]{APSS} it holds
\begin{equation}
\label{gauss1}
\mathsf{E}\,\|G_A:\,\ell_\infty^m\to \ell_1^n\|\asymp \|A\circ A:\,\ell_\infty^m\to \ell_{1/2}^n\|^{1/2}+\|(A\circ A)^T:\,\ell_\infty^n\to \ell_{1/2}^m\|^{1/2}
\end{equation}
%where
%$$
%D_1:=\|A\circ A:\,\ell_\infty^m\to \ell_{1/2}^n\|^{1/2},$$
%$$
%D_2:=\|(A\circ A)^T:\,\ell_\infty^n\to \ell_{1/2}^m\|^{1/2}$$
(here, $B^T$ is the transposed matrix for a matrix $B$).
Observe that
$$
\|A\circ A:\,\ell_\infty^m\to \ell_{1/2}^n\|^{1/2}=\sup_{\|t\|_{\ell_\infty^m}\leqslant 1}\sum_{i=1}^n\Big|\sum_{j=1}^m t_ja_{i,j}^2\Big|^{1/2}=\sum_{i=1}^n\Big(\sum_{j=1}^m a_{i,j}^2\Big)^{1/2}$$
and similarly 
$$
\|(A\circ A)^T:\,\ell_\infty^n\to \ell_{1/2}^m\|^{1/2}=\sum_{j=1}^m\Big(\sum_{i=1}^n a_{i,j}^2\Big)^{1/2}.$$
Thus, from \eqref{gauss1} it follows that

\begin{equation}
\label{gauss2}
\mathsf{E}\,\|G_A:\,\ell_\infty^m\to \ell_1^n\|\asymp \max\Big\{\sum_{i=1}^n\Big(\sum_{j=1}^m a_{i,j}^2\Big)^{1/2},\sum_{j=1}^m\Big(\sum_{i=1}^n a_{i,j}^2\Big)^{1/2}\Big\}. 
\end{equation}

Next, we equip the linear space $F_{n,m}$ of matrices $X=(x_{i,j})_{i\leqslant n, j\leqslant m}$  with the norm 
$$
{\|X\|}_F:=\|X:\,\ell_\infty^m\to \ell_1^n\|.$$
Let $X^{k,l}=(x_{i,j}^{k,l})_{i\leqslant n, j\leqslant m}$, where $x_{i,j}^{k,l}=a_{k,l}$ if $i=k$, $j=l$, and $x_{i,j}^{k,l}=0$, otherwise. Then, it is easy to see that
$$
G_A=\sum_{k=1}^n\sum_{l=1}^m g_{k,l}X^{k,l}\quad\mbox{and}\quad R_A=\sum_{k=1}^n\sum_{l=1}^m r_{k,l}X^{k,l}.$$
As is well known (see, e.g., \cite[Chapter 7]{LedTal} or \cite[Chapter 7]{AsBook})), the distribution of Banach space norms of random vector Gaussian sums majorizes (up to a constant) that of these norms of the corresponding Rademacher sums. In particular, we have 
$$
\mathsf{E}\|R_A\|_F\leqslant \sqrt{\frac{\pi}{2}}\cdot\mathsf{E}\|G_A\|_F$$ 
(cf. \cite[Corollary 7.1]{AsBook}). Combining this together with the definition of the norm ${\|X\|}_F$ and equivalence \eqref{gauss2}, we obtain the required result.
\end{proof}

\begin{proof}[The proof of Theorem \ref{theor_bi_Rad}] On the one hand, by Lemma \ref{Le1}, we arrive at the inequality
\begin{equation*}
\mathsf{E}_\theta\,\Bigl\|\sum_{i=1}^n\sum_{j=1}^m \theta_{i,j}a_{i,j}r_i\otimes r_j\Bigr\|_{L_\infty([0,1]^2)}\geqslant
\min_{\theta_{i,j}=\pm 1}{\Big\|\sum_{i=1}^n\sum_{j=1}^m \theta_{i,j}a_{i,j}r_i\otimes r_j\Big\|}_{L_\infty([0,1]^2)}
\end{equation*}
\begin{equation*}
\geqslant\frac{1}{\sqrt{2}}\max\Big\{\sum_{i=1}^n\Big(\sum_{j=1}^m a_{i,j}^2\Big)^{1/2},\sum_{j=1}^m\Big(\sum_{i=1}^n a_{i,j}^2\Big)^{1/2}\Big\}. 
\end{equation*}
On the other hand, with a universal constant $C>0$, from equation \eqref{eq_operator_norm} and Lemma \ref{Le2} it follows that 
\begin{eqnarray*}
\mathsf{E}_\theta\,\Bigl\|\sum_{i=1}^n\sum_{j=1}^m \theta_{i,j}a_{i,j}r_i\otimes r_j\Bigr\|_{L_\infty([0,1]^2)}&=&\mathsf{E}\,\|R_A:\,\ell_\infty^m\to \ell_1^n\|\\
%\end{equation*}
%\begin{equation*}
&\le& C\max\Big\{\sum_{i=1}^n\Big(\sum_{j=1}^m a_{i,j}^2\Big)^{1/2},\sum_{j=1}^m\Big(\sum_{i=1}^n a_{i,j}^2\Big)^{1/2}\Big\},
\end{eqnarray*}
%\geqslant \mathsf{E}\,\|R_A:\,\ell_\infty^m\to \ell_1^n\|
%\end{equation*}
%\begin{equation*}
%=\int\limits_0^1 \Bigl\|\sum_{i=1}^n\sum_{j=1}^m r_{i,j}(\omega) a_{i,j}r_i(u)r_j(v)\Bigr\|_{L_\infty([0,1]^2)}\,d\omega,
%\end{equation*}
and the proof is completed.
\end{proof}

The following result is an immediate consequence of Theorem \ref{theor_bi_Rad}.

\begin{cor}
\label{Theorema_RUC_mult_Rad}
The sequence $\{r_i\otimes r_j\}_{1\leqslant i<  \infty, 1\leqslant j<\infty}$ is a system of random unconditional convergence in the space $L_\infty([0,1]^2)$, i.e., there exists a universal constant $C>0$ such that for all $n,m\in\mathbb{N}$ and $a_{i,j}\in\mathbb{R}$   
$$
\mathsf{E}_\theta \Bigl\|\sum\nolimits_{i=1}^{n}\sum\nolimits_{j=1}^{m} \theta_{i,j}a_{i,j}r_i\otimes r_j\Bigr\|_{L_\infty([0,1]^2)}\leqslant C
{\Big\|\sum\nolimits_{i=1}^{n}\sum\nolimits_{j=1}^{m}a_{i,j}r_i\otimes r_j\Big\|}_{L_\infty([0,1]^2)}.
$$
%(here, $r_{i,j}$ is the Rademacher sequence, numbered in an arbitrary order).
\end{cor}

Moreover, in view of inequalities \eqref{eq_rad_cut_norm}, the statement of Theorem \ref{theor_bi_Rad} implies the analogous result for matrix cut-norms.
% of a random  matrix $(\theta_{i,j}a_{i,j})_{i\leqslant n, j\leqslant m}$.

\begin{cor}
\label{cor_cut_norm}
For every $n,m\in\mathbb{N}$ and all matrices $A=(a_{i,j})_{i\leqslant n, j\leqslant m}$, with universal constants it holds:
$$
\mathsf{E}_\theta{\|(\theta_{i,j}a_{i,j})\|}_{cut}\asymp \min_{\theta_{i,j}=\pm 1}{\|(\theta_{i,j}a_{i,j})\|}_{cut}\asymp \max\Big\{\sum_{i=1}^n\Big(\sum_{j=1}^m a_{i,j}^2\Big)^{1/2},\sum_{j=1}^m\Big(\sum_{i=1}^n a_{i,j}^2\Big)^{1/2}\Big\}. 
$$
\end{cor}

%{\color{red} Let $\mathcal{M}_{cut}$ be the linear space of all infinite matrices $A=(a_{i,j})_{1\le i,j<\infty}$ such that 
%$$
%\|A\|_{\mathcal{M}_{cut}}:=\sup_{n,m=1,2,\dots}{\|(a_{i,j})_{i \leqslant n, j\leqslant m}\|}_{cut}<\infty.$$ Then, Corollary \ref{cor_cut_norm} means that the standard unit matrices $E^{k,l}=(e^{k,l}_{i,j})_{1\le i,j<\infty}$, $1\le k,l<\infty$ (that is, $e^{k,l}_{k,l}=1$   and $e^{k,l}_{i,j}=0$ if either $i\ne k$ or $j\ne l$), form a RUC sequence in the Banach space $\mathcal{M}_{cut}$.}

%{\color{blue} 
Let $\mathcal{M}_{cut}$ be the linear space of all infinite matrices $A=(a_{i,j})_{1\le i,j<\infty}$ such that 
$$
\|A\|_{\mathcal{M}_{cut}}:=\sup_{n,m=1,2,\dots}{\|(a_{i,j})_{i \leqslant n, j\leqslant m}\|}_{cut}<\infty.$$
One can check that $\mathcal{M}_{cut}$ equipped with this norm is a Banach space. Furthermore, $\mathcal{M}_{cut}$  coincides with the space of all bounded linear operators from $c_0$ to $\ell_1$, where $c_0$ is the space of all sequences converging to zero \cite[Lemma 1]{Lorentz}. Let  $E^{k,l}=(e^{k,l}_{i,j})_{1\le i,j<\infty}$, $1\le k,l<\infty$, be  the sequence of standard unit matrices (that is, $e^{k,l}_{k,l}=1$,   and $e^{k,l}_{i,j}=0$ if either $i\ne k$ or $j\ne l$), which is ordered as follows: $E^{k_1,l_1}\prec E^{k_2,l_2}$ if $\max\{k_1,l_1\}<\max\{k_2,l_2\}$, or if $\max\{k_1,l_1\}=\max\{k_2,l_2\}$ and $k_1<k_2$, or if $\max\{k_1,l_1\}=\max\{k_2,l_2\}=k_1=k_2$ and $l_1>l_2$. Then, Corollary \ref{cor_cut_norm} means that $\{E^{k,l}\}_{1\leqslant k,l<\infty}$ is a RUC sequence in the space $\mathcal{M}_{cut}$. Moreover, it can be shown that the sequence $\{E^{k,l}\}_{1\leqslant k,l<\infty}$ is a RUC basis in $\mathcal{M}_{cut}$. 
%{\color{blue} 
%This gives another proof that any bounded linear operator from $c_0$ to $\ell_1$ is compact (a special case of the famous Pitt theorem, see \cite[Proposition 2.c.3]{LT1}).}
%A similar result for Schatten-von Neumann ideals $S^p$, $2\le p<\infty$, was proved by Pisier (1985). 

  \section{RUC-property of the second-order Rademacher chaos in $L_\infty$}
\label{Sec_chaos_Rad}

In this section we prove, by using the decoupling techniques, a result similar to Theorem \ref{theor_bi_Rad} for the second-order Rademacher chaos.

%С помощью этой теоремы мы просто сведем утверждения о (неразделенном) хаосе к результатам о кратной системе Радемахера из предыдущего раздела.
\begin{theorem}\label{theor_chaos}
With universal constants, for all $n\in\mathbb{N}$ and $a_{i,j}\in\mathbb{R}$, $1\leqslant i<j\leqslant n$, we have
\begin{equation*}
\mathsf{E}\,_\theta\Bigl\|\sum_{i=1}^n\sum_{j=i+1}^n \theta_{i,j}a_{i,j}r_i r_j\Bigr\|_{L_\infty}\asymp
\min_{\theta_{i,j}=\pm 1}{\Big\|\sum_{i=1}^n\sum_{j=i+1}^n \theta_{i,j}a_{i,j}r_i r_j\Big\|}_{L_\infty}
\end{equation*}
\begin{equation*}
\asymp\max\Big\{\sum_{i=1}^{n-1}\Big(\sum_{j=i+1}^n a_{i,j}^2\Big)^{1/2},\sum_{j=2}^n\Big(\sum_{i=1}^{j-1} a_{i,j}^2\Big)^{1/2}\Big\}. 
\end{equation*}
\end{theorem}
\begin{proof} Setting $b_{i,j}=a_{i,j}/2$ if $i<j$, $b_{i,j}=b_{j,i}$ if $i>j$ and $b_{i,i}=0$, and applying Corollary \ref{real theor_decoupling} in the case $d=2$, we have 
%{\color{red}
%\begin{eqnarray*}
%\frac{1}{12}\Bigl\|\sum_{i=1}^n\sum_{j=1}^n b_{i,j}r_ir_j\Bigr\|_{L_\infty}&\leqslant& \Bigl\|\sum_{i=1}^n\sum_{j=1}^n b_{i,j}r_i\otimes r_j\Bigr\|_{L_\infty([0,1]^2)}\\&\leqslant& 4 \Bigl\|\sum_{i=1}^n\sum_{j=1}^n b_{i,j}r_ir_j\Bigr\|_{L_\infty}.
%\end{eqnarray*}}
%{\color{blue}
$$
\Bigl\|\sum_{i=1}^n\sum_{j=1}^n b_{i,j}r_i\otimes r_j\Bigr\|_{L_\infty([0,1]^2)}\leqslant 4 \Bigl\|\sum_{i=1}^n\sum_{j=1}^n b_{i,j}r_ir_j\Bigr\|_{L_\infty}.
$$
Hence, by Lemma \ref{Le1}
\begin{eqnarray*}
\Bigl\|\sum_{i=1}^n\sum_{j=1}^n b_{i,j}r_ir_j\Bigr\|_{L_\infty}&\geqslant&  \frac{1}{4\sqrt{2}}\max\Big\{\sum_{i=1}^n\Big(\sum_{j=1}^n b_{i,j}^2\Big)^{1/2},\sum_{j=1}^n\Big(\sum_{i=1}^{n} b_{i,j}^2\Big)^{1/2}\Big\}\\&=&\frac{1}{4\sqrt{2}}\sum_{i=1}^n\Big(\sum_{j=1}^n b_{i,j}^2\Big)^{1/2}.
\end{eqnarray*}
Since
\begin{eqnarray*}
\sum_{i=1}^n\Big(\sum_{j=1}^n b_{i,j}^2\Big)^{1/2}&\geqslant& \max\Big\{\sum_{i=1}^{n}\Big(\sum_{j=i+1}^n b_{i,j}^2\Big)^{1/2}, \sum_{i=1}^n\Big(\sum_{j=1}^{i-1} b_{i,j}^2\Big)^{1/2}\Big\}\\
&=&\frac{1}{4} \max\Big\{\sum_{i=1}^{n-1}\Big(\sum_{j=i+1}^n a_{i,j}^2\Big)^{1/2}, \sum_{j=2}^n\Big(\sum_{i=1}^{j-1} a_{i,j}^2\Big)^{1/2}\Big\}
\end{eqnarray*}
and, by the definition of coefficients $b_{i,j}$,
$$
\sum_{i=1}^n\sum_{j=1}^n b_{i,j}r_i(t)r_j(t)=\sum_{i=1}^n\sum_{j=i+1}^n a_{i,j}r_i(t) r_j(t),\;\;0\le t\le 1,$$
this implies that
\begin{equation}
\label{below est}
\Bigl\|\sum_{i=1}^n\sum_{j=i+1}^n a_{i,j}r_i r_j\Bigr\|_{L_\infty}\ge\frac{1}{16\sqrt{2}}\max\Big\{\sum_{i=1}^{n-1}\Big(\sum_{j=i+1}^n a_{i,j}^2\Big)^{1/2}, \sum_{j=2}^n\Big(\sum_{i=1}^{j-1} a_{i,j}^2\Big)^{1/2}\Big\}.
\end{equation}

On the other hand, obviously, for any $c_{i,j}\in\mathbb{R}$
$$
\Bigl\|\sum_{i=1}^n\sum_{j=1}^n c_{i,j}r_i r_j\Bigr\|_{L_\infty}\le \Bigl\|\sum_{i=1}^n\sum_{j=1}^n c_{i,j}r_i\otimes r_j\Bigr\|_{L_\infty([0,1]^2)}.$$
Therefore, by Theorem \ref{theor_bi_Rad},
\begin{eqnarray*}
\mathsf{E}\,_\theta\Bigl\|\sum_{i=1}^n\sum_{j=i+1}^n \theta_{i,j}a_{i,j}r_i r_j\Bigr\|_{L_\infty}&\le&\mathsf{E}\,_\theta\Bigl\|\sum_{i=1}^n\sum_{j=i+1}^n \theta_{i,j}a_{i,j}r_i\otimes r_j\Bigr\|_{L_\infty([0,1]^2)}\\&\le&
C\max\Big\{\sum_{i=1}^{n-1}\Big(\sum_{j=i+1}^n a_{i,j}^2\Big)^{1/2}, \sum_{j=2}^n\Big(\sum_{i=1}^{j-1} a_{i,j}^2\Big)^{1/2}\Big\}
\end{eqnarray*}
with some universal constant $C>0$. Combining this estimate with inequality \eqref{below est}, we complete the proof.

\end{proof}

Similarly, as for the multiple Rademacher system (see Corollary \ref{cor_cut_norm} in Section \ref{Sec_mult_Rad}), from Theorem \ref{theor_chaos} and equivalence \eqref{eq_square_cut_norm} we infer the following result for modified matrix $cut$-norms.

\begin{cor}
\label{cor_square_norm}
For each strictly upper triangular matrix $A=(a_{i,j})_{1\le i<j\leqslant n}$ with universal  constants it holds:
$$
\mathsf{E}_\theta{\|(\theta_{i,j}a_{i,j})\|}_{cut}^*\asymp \min_{\theta_{i,j}=\pm 1}{\|(\theta_{i,j}a_{i,j})\|}_{cut}^*\asymp \max\Big\{\sum_{i=1}^{n-1}\Big(\sum_{j={i+1}}^n a_{i,j}^2\Big)^{1/2},\sum_{j=2}^n\Big(\sum_{i=1}^{j-1} a_{i,j}^2\Big)^{1/2}\Big\}. 
$$
\end{cor}

\section{RUC-property of the multiple Rademacher system and the Rademacher chaos of arbitrary order in $L_\infty$}
\label{Sec_mult_d_Rad}

Here, we extend the results proved in two previous sections to the systems of an arbitrary multiplicity.

%\begin{eqnarray}
%\label{upper est}
%&\;&\int_0^1\Big\|\sum_{\vec{j}\in \mathbb{N}^{d}_{n}}a_{\vec{j}}r_{\vec{j}}(\omega){\rm r}_{\vec{j}}^{\otimes}(\vec{t})\Big\|_{L_\infty([0,1]^d}\,d\omega\nonumber\\&\leqslant & \sum_{j_{d}=1}^n\Big(\sum_{\vec{j_d}'\in \mathbb{N}^{d-1}_{n}}a_{\vec{j}}^2\Big)^{1/2}+\dots + 2^{d-1}\cdot\sum_{j_{1}=1}^n\Big(\sum_{\vec{j_1}'\in \mathbb{N}^{d-1}_{n}}a_{\vec{j}}^2\Big)^{1/2}.
%\end{eqnarray}

We introduce first some additional notation. Let $d\in \mathbb{N}$. $d\ge 2$.
% we let
%$$
%\mathbb{N}^d:=\{{\vec{j}}=(j_1,j_{2},\ldots,j_d):\, j_k\in\mathbb{N},\, k=1,2,\ldots,d \},
%$$
%$$
%\mathbb{N}_n^d:=\{{\vec{j}}=(j_1,j_{2},\ldots,j_d)\in\mathbb{N}^d:\, j_k\leqslant n,\, k=1,2,\ldots,d \},
%$$
%and, for ${\vec{j}}=(j_1,j_{2},\ldots,j_d)\in\mathbb{N}^d$,
%$$
%{\rm r}_{{\vec{j}}}^{\otimes}(\vec{t})=r_{j_1}(t_1)\cdot r_{j_2}(t_2)\cdot\dots\cdot r_{j_d}(t_d),\;\;\mbox{where}\;\vec{t}=(t_1,t_2,\dots,t_d)\in [0,1]^d.$$
%Next, for each $k=1,\dots,d$,} {\color{blue} 
For every $\vec{j}\in\mathbb{N}^d$, $k=1,2,\ldots, d$ and $t\in[0,1]^d$ we define ${\vec{j_k'}}\in \mathbb{N}^{d-1}$ and ${\vec{t_k'}}\in [0,1]^{d-1}$ as follows:
$$
{\vec{j_k'}}=(j_1,\ldots,j_{k-1},j_{k+1},\ldots,j_d)\;\;\mbox{and}\;\;{\vec{t_k'}}=(t_1,\ldots,t_{k-1},t_{k+1},\ldots,t_d)$$
and
$$
{\rm r}_{{\vec{j_k}}'}^{\otimes}({\vec{t}_k}'):=r_{j_1}(t_1)\cdot \dots\cdot r_{j_{k-1}}(t_{k-1})\cdot r_{j_{k+1}}(t_{k+1})\cdot\dots\cdot r_{j_d}(t_d).
$$
Moreover, for every $d,n\in\mathbb{N}$, $k=1,2,\ldots, d$ and $l=1,2,\ldots, n$ we put
$$
\mathbb{N}_n^d(k,l):=\{\vec{j}=(j_1,\dots,j_d)\in \mathbb{N}_n^d:\,j_{k}=l\}.$$

\begin{theorem}
\label{theor_mult_Rad}
For every $d\in \mathbb{N}$ the multiple Rademacher sequence $\{{\rm r}_{{\vec{j}}}^{\otimes}\}_{\vec{j}\in \mathbb{N}^d}$ has the RUC property in the space $L_\infty([0,1]^d)$. More precisely, for all $n\in\mathbb{N}$ and $a_{\vec{j}}\in\mathbb{R}$, $\vec{j}\in\mathbb{N}_n^d$, the following inequalities hold
\begin{equation}
\label{lower est}
\Big\|\sum_{\vec{j}\in \mathbb{N}^{d}_{n}}a_{\vec{j}} {\rm r}_{\vec{j}}^{\otimes}\Big\|_{L_\infty([0,1]^d)}\geqslant c_d\cdot\max_{k\in [d]}\sum_{l=1}^n \Big(\sum_{{\vec{j}\in \mathbb{N}^{d}_{n}(k,l)}}a_{{\vec{j}}}^2\Big)^{1/2},
\end{equation}
where the constant $c_d$ depends only on $d$, and
\begin{eqnarray}
\label{upper est}
\mathsf{E}_{\theta}\Big\|\sum_{\vec{j}\in \mathbb{N}^{d}_{n}}a_{\vec{j}}\theta_{\vec{j}}{\rm r}_{\vec{j}}^{\otimes}\Big\|_{L_\infty([0,1]^d)}&\leqslant& \sum_{l=1}^n\Big(\sum_{{\vec{j}\in \mathbb{N}^{d}_{n}(d,l)}}a_{\vec{j}}^2\Big)^{1/2}+\dots\nonumber\\ &+& 2^{d-1}\cdot\sum_{l=1}^n\Big(\sum_{{\vec{j}\in \mathbb{N}^{d}_{n}(1,l)}}a_{\vec{j}}^2\Big)^{1/2}.
\end{eqnarray}

%\begin{eqnarray}
%\label{upper est}
%&\;&\int_0^1\Big\|\sum_{\vec{j}\in \mathbb{N}^{d}_{n}}a_{\vec{j}}r_{\vec{j}}(\omega){\rm r}_{\vec{j}}^{\otimes}(\vec{t})\Big\|_{L_\infty([0,1]^d}\,d\omega\nonumber\\&\leqslant & \sum_{j_{d}=1}^n\Big(\sum_{\vec{j_d}'\in \mathbb{N}^{d-1}_{n}}a_{\vec{j}}^2\Big)^{1/2}+\dots + 2^{d-1}\cdot\sum_{j_{1}=1}^n\Big(\sum_{\vec{j_1}'\in \mathbb{N}^{d-1}_{n}}a_{\vec{j}}^2\Big)^{1/2}.
%\end{eqnarray}

\end{theorem}

\begin{proof}
First, we prove estimate \eqref{lower est}. 

Let $k\in [d]$ be arbitrary.
%Let
%$$
%S_{a,n}(\vec{t}):=\sum_{\vec{j}\in \mathbb{N}^{d}_{n}}a_{\vec{j}} {\rm r}_{\vec{j}}^{\otimes}(\vec{t}),\;\;{\vec t}\in [0,1]^{d},
%$$
%and let $k\in [d]$ be arbitrary. 
By choosing $t_{k}\in[0,1]$ properly, we obtain
$$
\Big\|\sum_{\vec{j}\in \mathbb{N}^{d}_{n}}a_{\vec{j}} {\rm r}_{\vec{j}}^{\otimes}\Big\|_{L_\infty([0,1]^d)} =\sup_{{\vec{t}}\in [0,1]^{d}}\Big\vert\sum_{\vec{j}\in \mathbb{N}^{d}_{n}}a_{\vec{j}} {\rm r}_{{\vec{j}}}^{\otimes}({\vec{t}})\Big\vert=
\sup_{{\vec{t}_k}'\in [0,1]^{d-1}}\sum_{l=1}^n \Big\vert\sum_{\vec{j}\in \mathbb{N}^{d}_{n}(k,l)}a_{\vec{j}} {\rm r}_{{\vec{j_k}}'}^{\otimes}({\vec{t}_k}')\Big\vert.
$$
Hence, by Bonami's inequality \eqref{Bonami} for the case $p=1$, we infer
%или, в варианте неразделенного хаоса, \cite[лемма 6]{AL_AA}),
% а в случае $d=2$, предложение~2.d.6 в \cite{LT} или задачу IX.3.6 в \cite{MLP})  
\begin{eqnarray}
\Big\|\sum_{\vec{j}\in \mathbb{N}^{d}_{n}}a_{\vec{j}} {\rm r}_{\vec{j}}^{\otimes}\Big\|_{L_\infty([0,1]^d)} &\geqslant &\int_0^1\ldots\int_0^1\sum_{l=1}^n \Big\vert\sum_{\vec{j}\in \mathbb{N}^{d}_{n}(k,l)}a_{\vec{j}}{\rm r}_{{\vec{j_k}}'}^{\otimes}({\vec{t}_k}')\Big\vert \,dt_1\ldots \,dt_{k-1}\,dt_{k+1}\ldots \,dt_d\nonumber\\
%&=&\sum_{j_{k}=1}^n\int_0^1\ldots\int_0^1\Big\vert\sum_{\vec{j}\in \mathbb{N}^{d}_{n}(k,l)}a_{\vec{j}} {\rm r}_{{\vec{j_k}}'}^{\otimes}({\vec{t}_k}')\Big\vert \,dt_1dt_2\ldots \,dt_{k-1}\,dt_{k+1}\ldots dt_d\nonumber\\
&\geqslant& A_{1,d}\sum_{l=1}^n \Big(\sum_{{\vec{j}\in \mathbb{N}^{d}_{n}(k,l)}}a_{{\vec{j}}}^2\Big)^{1/2},\;\;k=1,2,\dots,d,
\label{first est}
\end{eqnarray}
which implies inequality \eqref{lower est}.

Let us now proceed with the proof of estimate \eqref{upper est}. Further, for technical reasons, it will be more convenient to deal with averaging generated by Rademacher functions rather than that over arrangements of signs. To this end, we observe that
\begin{equation}
\label{extra not}
\mathsf{E}\,_{\theta}\Big\|\sum_{\vec{j}\in \mathbb{N}^{d}_{n}}a_{\vec{j}}\theta_{\vec{j}}{\rm r}_{\vec{j}}^{\otimes}\Big\|_{L_\infty([0,1]^d)}=
\mathsf{E}_\omega\|S_{n}(\cdot,\omega)\|_{L_\infty([0,1]^d)},
\end{equation}
where
$$
S_{n}(\vec{t},\omega):=\sum_{\vec{j}\in \mathbb{N}^{d}_{n}}a_{\vec{j}}r_{\vec{j}}(\omega){\rm r}_{\vec{j}}^{\otimes}(\vec{t}),\;\;\vec{t}\in [0,1]^d,\,\omega\in [0,1]$$
(here, $r_{\vec{j}}(\omega)$ are the first $n^d$ Rademacher functions enumerated arbitrarily by multi-indices $\vec{j}\in\mathbb{N}_n^d$).

Next, denoting $z_{\vec{j}}(\omega):=a_{\vec{j}}r_{\vec{j}}(\omega)$ if $\vec{j}=(j_1,\dots,j_d)\in \mathbb{N}^{d}_{n}$ and $\omega\in[0,1]$, we have
%$$
%\int_0^1\Big\|\sum_{\vec{j}\in \mathbb{N}^{d}_{n}}a_{\vec{j}}r_{\vec{j}}(\omega){\rm r}_{\vec{j}}^{\otimes}\Big\|_{L_\infty([0,1]^d)}\,d\omega,$$
%where $\{r_{\vec{j}}(\omega)\}_{\vec{j}\in\mathbb{N}_n^d}$ is the sequence of the first $n^d$ Rademacher functions, arbitrarily numbered by the vectors $\vec{j}\in\mathbb{N}_n^d$.
\begin{equation*}
{\|S_{n}(\cdot,\omega)\|}_{L_\infty([0,1]^d)}=\max_{x^i\in\mathcal{E}}\sum_{\vec{j}\in \mathbb{N}^{d}_{n}}x_{j_1}^1x_{j_2}^2\dots x_{j_d}^d z_{\vec{j}}(\omega),
\end{equation*}
%\begin{eqnarray*}
%{\|S_{n}(\vec{t},\omega)\|}_{L_\infty([0,1]^d)}&=&\sup_{\|x^i\|_{\ell_\infty^n}\leqslant 1,i\leqslant d}\sum_{\vec{j}\in \mathbb{N}^{d}_{n}}z_{\vec{j}}x_{j_1}^1x_{j_2}^2\dots x_{j_d}^d\\
%&=&\max_{x^i\in\mathcal{E},i\leqslant d}\sum_{\vec{j}\in \mathbb{N}^{d}_{n}}z_{\vec{j}}x_{j_1}^1x_{j_2}^2\dots x_{j_d}^d,
%\end{eqnarray*}
where each of the vectors $x^i=(x_l^i)_{l=1}^n$, $i=1,2,\dots,d$, runs through the set $\mathcal{E}$ of all arrangements of signs, i.e., $x_l^i=\pm 1$, $l=1,2,\dots,n$.
%the vectors of the form $(\eta_1, \eta_2,\ldots \eta_n)$, $\eta_j=\pm 1$, $j=1,2,\dots,n$. 
Hence,
\begin{equation}\label{EQ1}
\mathsf{E}_\omega\|S_{n}(\cdot,\omega)\|_{L_\infty([0,1]^d)}=\mathsf{E}_\omega \max_{x^i\in\mathcal{E}}\sum_{\vec{j}\in \mathbb{N}^{d}_{n}}x_{j_1}^1x_{j_2}^2\dots x_{j_d}^d z_{\vec{j}}(\omega),
\end{equation}
or, as above,
$$
\mathsf{E}_\omega\|S_{n}(\cdot,\omega)\|_{L_\infty([0,1]^d)}=\mathsf{E}_\omega \max_{x^i\in\mathcal{E}}\sum_{l=1}^n\Big|\sum_{\vec{j}\in \mathbb{N}^{d}_{n}(d,l)}x_{j_1}^1x_{j_2}^2\dots x_{j_{d-1}}^{d-1}z_{\vec{j}}(\omega)\Big|.$$
%$$
%\mathsf{E}_\omega\|S_{a,n}(\vec{t},\omega)\|_{L_\infty([0,1]^d)}=\mathsf{E}_\omega \max_{x^i\in\mathcal{E},i\leqslant d-1}\sum_{l=1}^n\Big|\sum_{\vec{j_d}'\in \mathbb{N}^{d-1}_{n}}z_{\vec{j}}x_{j_1}^1x_{j_2}^2\dots x_{j_{d-1}}^{d-1}\Big|.$$
Centering the expression from the right-hand side, we obtain 
%{\color{red}
%\begin{eqnarray}
%\label{general est}
%{\rm Ave}&\leqslant& \mathsf{E}_\omega \max_{x^i\in\mathcal{E},i\leqslant d-1}\sum_{l=1}^n\left(\Big|\sum_{\vec{j_d}'\in \mathbb{N}^{d-1}_{n}}x_{j_1}^1x_{j_2}^2\dots x_{j_{d-1}}^{d-1}z_{\vec{j}}(\omega)\Big|-\mathsf{E}_\omega\Big|\sum_{\vec{j}\in \mathbb{N}^{d-1}_{n}}x_{j_1}^1x_{j_2}^2\dots x_{j_{d-1}}^{d-1}z_{\vec{j}}(\omega)\Big|\right)\nonumber\\
%\label{general est}
%&+&\max_{x^i\in\mathcal{E},i\leqslant d-1}\sum_{l=1}^n\mathsf{E}_\omega\Big|\sum_{\vec{j_d}'\in \mathbb{N}^{d-1}_{n}}x_{j_1}^1x_{j_2}^2\dots x_{j_{d-1}}^{d-1}z_{\vec{j}}(\omega)\Big|=(I)+(II).
%\label{general est}
%\end{eqnarray}
%}
%{\color{blue}
\begin{equation}
\label{general est}
\mathsf{E}_\omega\|S_{n}(\cdot,\omega)\|_{L_\infty([0,1]^d)}\leqslant A_1+A_2,
\end{equation}
where
$$
A_1:=\mathsf{E}_\omega \max_{x^i\in\mathcal{E}}\sum_{l=1}^n\biggl(\Big|\sum_{\vec{j}\in \mathbb{N}^{d}_{n}(d,l)}x_{j_1}^1x_{j_2}^2\dots x_{j_{d-1}}^{d-1}z_{\vec{j}}(\omega)\Big|-\mathsf{E}_\omega\Big|\sum_{\vec{j}\in \mathbb{N}^{d}_{n}(d,l)}x_{j_1}^1x_{j_2}^2\dots x_{j_{d-1}}^{d-1}z_{\vec{j}}(\omega)\Big|\biggr)
$$
and
$$
A_2:=\max_{x^i\in\mathcal{E}}\sum_{l=1}^n\mathsf{E}_\omega\Big|\sum_{\vec{j}\in \mathbb{N}^{d}_{n}(d,l)}x_{j_1}^1x_{j_2}^2\dots x_{j_{d-1}}^{d-1}z_{\vec{j}}(\omega)\Big|.$$

Let us first estimate the term $A_2$. Since the system $\{r_{\vec{j}}\}$ is orthonormal, for each $l=1,2,\dots,n$ we obtain
\begin{eqnarray*}
\mathsf{E}_\omega\Big|\sum_{\vec{j}\in \mathbb{N}^{d}_{n}(d,l)}x_{j_1}^1x_{j_2}^2\dots x_{j_{d-1}}^{d-1}z_{\vec{j}}(\omega)\Big| &=&
\mathsf{E}_\omega\Big|\sum_{\vec{j}\in \mathbb{N}^{d}_{n}(d,l)}a_{\vec{j}}x_{j_1}^1x_{j_2}^2\dots x_{j_{d-1}}^{d-1}r_{\vec{j}}(\omega)\Big|\nonumber\\ &\leqslant &
\Big(\sum_{\vec{j}\in \mathbb{N}^{d}_{n}(d,l)}(a_{\vec{j}}x_{j_1}^1x_{j_2}^2\dots x_{j_{d-1}}^{d-1})^2\Big)^{1/2}\nonumber\\&=&
\Big(\sum_{\vec{j}\in \mathbb{N}^{d}_{n}(d,l)}a_{\vec{j}}^2\Big)^{1/2},
\end{eqnarray*}
which implies that
\begin{equation}
A_2\leqslant \sum_{l=1}^n\Big(\sum_{\vec{j}\in \mathbb{N}^{d}_{n}(d,l)}a_{\vec{j}}^2\Big)^{1/2}.
\label{general-1 est}
\end{equation}

%\begin{equation}
%\max_{x^i\in\mathcal{E},i\leqslant d-1}\sum_{l=1}^n\mathsf{E}_\omega\Big|\sum_{\vec{j_d}'\in \mathbb{N}^{d-1}_{n}}z_{\vec{j}}x_{j_1}^1x_{j_2}^2\dots x_{j_{d-1}}^{d-1}\Big| \leqslant \sum_{l=1}^n\Big(\sum_{\vec{j_d}'\in \mathbb{N}^{d-1}_{n}}a_{\vec{j}}^2\Big)^{1/2}.
%\label{general-1 est}
%\end{equation}

To estimate $A_1 $, we apply a symmetrization trick and the well-known Talagrand's contraction principle. Denote by $\tilde{z}_{\vec{j}}$ independent copies of the random variables $z_{\vec{j}}$, $\vec{j}\in \mathbb{N}^{d}_{n}$. We can assume that $\tilde{z}_{\vec{j}}={z}_{\vec{j}}(\tilde{\omega})$, where $(\omega,\tilde{\omega})\in [0,1]^2$. By $(\epsilon_{l})_{l=1}^n$ we denote a sequence of "random"\:signs (that is, a sequence of symmetrically and identically distributed independent random variables, taking on the values $+1$ and $-1$ with probability $1/2$), which are independent with respect to all other random variables involved. 
%{\color{red} Observe that the vectors
%$$
%\left\{\epsilon_{j_d}\Biggl(\Big|\sum_{\vec{j_d}'\in \mathbb{N}^{d-1}_{n}}x_{j_1}^1x_{j_2}^2\dots x_{j_{d-1}}^{d-1}{z}_{\vec{j}}({\omega})\Big|-\Big|\sum_{\vec{j_d}'\in \mathbb{N}^{d-1}_{n}}x_{j_1}^1x_{j_2}^2\dots x_{j_{d-1}}^{d-1}{z}_{\vec{j}}(\tilde{\omega})\Big|\Biggr)\right\}_{x^i\in\mathcal{E}},\;j_d\in[n],$$
%are symmetric and, for each $j_d=1,2,\dots,n$, have the same distribution as the vector
%$$
%\left\{\Big|\sum_{\vec{j_d}'\in \mathbb{N}^{d-1}_{n}}x_{j_1}^1x_{j_2}^2\dots x_{j_{d-1}}^{d-1}z_{\vec{j}}(\omega)\Big|-\Big|\sum_{\vec{j_d}'\in \mathbb{N}^{d-1}_{n}}x_{j_1}^1x_{j_2}^2\dots x_{j_{d-1}}^{d-1}{z}_{\vec{j}}(\tilde{\omega})\Big|\right\}_{x^i\in\mathcal{E}}.$$
%Moreover, these vectors are independent, because 
%$$
%\{\vec{j}\in \mathbb{N}^{d}_{n}:\,j_d=k\}\cap \{\vec{j}\in \mathbb{N}^{d}_{n}:\,j_d=l\}=\emptyset\;\;\mbox{if}\,1\le k\ne l\le n.$$
%Therefore, by the inequalities}
For each $l=1,2,\ldots, n$, we consider the $2^{n(d-1)}$-dimensional random vectors
$$
v_{l}:=\left\{\epsilon_{l}\Biggl(\Big|\sum_{\vec{j}\in \mathbb{N}^{d}_{n}(d,l)}x_{j_1}^1x_{j_2}^2\dots x_{j_{d-1}}^{d-1}{z}_{\vec{j}}({\omega})\Big|-\Big|\sum_{\vec{j}\in \mathbb{N}^{d}_{n}(d,l)}x_{j_1}^1x_{j_2}^2\dots x_{j_{d-1}}^{d-1}{z}_{\vec{j}}(\tilde{\omega})\Big|\Biggr)\right\}_{x^i\in\mathcal{E},i\leqslant d-1}$$
and
$$u_{l}:=\left\{\Big|\sum_{\vec{j}\in \mathbb{N}^{d}_{n}(d,l)}x_{j_1}^1x_{j_2}^2\dots x_{j_{d-1}}^{d-1}z_{\vec{j}}(\omega)\Big|-\Big|\sum_{\vec{j}\in \mathbb{N}^{d}_{n}(d,l)}x_{j_1}^1x_{j_2}^2\dots x_{j_{d-1}}^{d-1}{z}_{\vec{j}}(\tilde{\omega})\Big|\right\}_{x^i\in\mathcal{E}, i\leqslant d-1}.$$
As is easy to see, for every $l=1,2,\ldots, n$, the vectors $v_{l}$ and $u_{l}$ are symmetrically distributed and have the same distribution.
%$$
%u_{j_d}:=\left\{\Big|\sum_{\vec{j_d}'\in \mathbb{N}^{d-1}_{n}}x_{j_1}^1x_{j_2}^2\dots x_{j_{d-1}}^{d-1}z_{\vec{j}}(\omega)\Big|-\Big|\sum_{\vec{j_d}'\in \mathbb{N}^{d-1}_{n}}x_{j_1}^1x_{j_2}^2\dots x_{j_{d-1}}^{d-1}{z}_{\vec{j}}(\tilde{\omega})\Big|\right\}_{x^i\in\mathcal{E}, i\leqslant d-1}.$$
Moreover, each of the sequences $\{v_{l}\}_{l=1}^n$ and $\{u_{l}\}_{l=1}^n$ consists of independent random variables, because $z_{\vec{j}}(\omega)$ and ${z}_{\vec{j}}(\tilde{\omega})$, $\vec{j}\in \mathbb{N}^{d}_{n}$, are independent and  
$$
\mathbb{N}^{d}_{n}(d,k)\cap \mathbb{N}^{d}_{n}(d,l)=\emptyset\;\;\mbox{if}\;1\le k\ne l\le n.$$
%$$
%\{\vec{j}\in \mathbb{N}^{d}_{n}:\,j_d=k\}\cap \{\vec{j}\in \mathbb{N}^{d}_{n}:\,j_d=l\}=\emptyset\;\;\mbox{if}\,1\le k\ne l\le n.$$
%Similarly, vectors  $u_{1}, u_2,\ldots, u_n$ are independent. 
Consequently, the vectors $\sum_{l=1}^n v_{l}$ and $\sum_{l=1}^n u_{l}
$ also
%$$
%\sum_{l=1}^n v_{j_d}\quad\mbox{and}\quad \sum_{l=1}^n u_{j_d}
%$$
have the same distribution.
Hence, by using the inequalities
$$
\sup_\alpha \mathsf{E}\xi_\alpha\leqslant \mathsf{E}\sup_\alpha \xi_\alpha\quad\mbox{ and }\quad\sup_\alpha (\xi_\alpha +\eta_\alpha)\leqslant \sup_\alpha \xi_\alpha +\sup_\alpha \eta_\alpha,
$$
which hold for arbitrary families of random variables $\{\xi_\alpha\}_\alpha$ and $\{\eta_\alpha\}_\alpha$, we obtain
%где $J_1\subset \mathbb{N}^{1}_{n}$, $J_2\subset \mathbb{N}^{1}_{n}$, $J_1\cap J_2=\emptyset$, получим 
%{\color{red}
%\begin{eqnarray}
%(I) &=&\mathsf{E}_\omega \max_{x^i\in\mathcal{E}}\sum_{l=1}^n\left(\Big|\sum_{\vec{j_d}'\in \mathbb{N}^{d-1}_{n}}x_{j_1}^1x_{j_2}^2\dots x_{j_{d-1}}^{d-1}{z}_{\vec{j}}({\omega})\Big|-\mathsf{E}_{\tilde{\omega}}\Big|\sum_{\vec{j_d}'\in \mathbb{N}^{d-1}_{n}}x_{j_1}^1x_{j_2}^2\dots x_{j_{d-1}}^{d-1}{z}_{\vec{j}}(\tilde{\omega})\Big|\right)\nonumber\\&\leqslant&\mathsf{E}_{\omega,\tilde{\omega}}\max_{x^i\in\mathcal{E}}\sum_{l=1}^n \Big(\Big|\sum_{\vec{j_d}'\in \mathbb{N}^{d-1}_{n}}{z}_{\vec{j}}({\omega})x_{j_1}^1x_{j_2}^2\dots x_{j_{d-1}}^{d-1}\Big|-\Big|\sum_{\vec{j_d}'\in \mathbb{N}^{d-1}_{n}}x_{j_1}^1x_{j_2}^2\dots x_{j_{d-1}}^{d-1}{z}_{\vec{j}}(\tilde{\omega})\Big|\Big)\nonumber\\&\leqslant&
%\mathsf{E}_{\omega,\tilde{\omega},\epsilon} \max_{x^i\in\mathcal{E}}\sum_{l=1}^n \epsilon_{j_d}\left(\Big|\sum_{\vec{j_d}'\in \mathbb{N}^{d-1}_{n}}x_{j_1}^1x_{j_2}^2\dots x_{j_{d-1}}^{d-1}{z}_{\vec{j}}({\omega})\Big|-\Big|\sum_{\vec{j_d}'\in \mathbb{N}^{d-1}_{n}}x_{j_1}^1x_{j_2}^2\dots x_{j_{d-1}}^{d-1}{z}_{\vec{j}}(\tilde{\omega})\Big|\right)\nonumber\\&\leqslant& 2\cdot \mathsf{E}_{\omega,\epsilon} \max_{x^i\in\mathcal{E}}\sum_{l=1}^n \epsilon_{j_d}\Big|\sum_{\vec{j_d}'\in \mathbb{N}^{d-1}_{n}}x_{j_1}^1x_{j_2}^2\dots x_{j_{d-1}}^{d-1}{z}_{\vec{j}}({\omega})\Big|.
%\label{long ineq}
%\end{eqnarray}
%}
%{\color{blue}
\begin{eqnarray}
A_1 &=&\mathsf{E}_\omega \max_{x^i\in\mathcal{E}}\sum_{l=1}^n\biggl(\Big|\sum_{\vec{j}\in \mathbb{N}^{d}_{n}(d,l)}x_{j_1}^1x_{j_2}^2\dots x_{j_{d-1}}^{d-1}{z}_{\vec{j}}({\omega})\Big|-\mathsf{E}_{\tilde{\omega}}\Big|\sum_{\vec{j}\in \mathbb{N}^{d}_{n}(d,l)}x_{j_1}^1x_{j_2}^2\dots x_{j_{d-1}}^{d-1}{z}_{\vec{j}}(\tilde{\omega})\Big|\biggr)\nonumber\\
&\leqslant&\mathsf{E}_{\omega,\tilde{\omega}}\max_{x^i\in\mathcal{E}}\sum_{l=1}^n \Big(\Big|\sum_{\vec{j}\in \mathbb{N}^{d}_{n}(d,l)}x_{j_1}^1x_{j_2}^2\dots x_{j_{d-1}}^{d-1}{z}_{\vec{j}}({\omega})\Big|-\Big|\sum_{\vec{j}\in \mathbb{N}^{d}_{n}(d,l)}x_{j_1}^1x_{j_2}^2\dots x_{j_{d-1}}^{d-1}{z}_{\vec{j}}(\tilde{\omega})\Big|\Big)\nonumber\\
&=&
\mathsf{E}_{\omega,\tilde{\omega},\epsilon} \max_{x^i\in\mathcal{E}}\sum_{l=1}^n \epsilon_{l}\biggl(\Big|\sum_{\vec{j}\in \mathbb{N}^{d}_{n}(d,l)}x_{j_1}^1x_{j_2}^2\dots x_{j_{d-1}}^{d-1}{z}_{\vec{j}}({\omega})\Big|-\Big|\sum_{\vec{j}\in \mathbb{N}^{d}_{n}(d,l)}x_{j_1}^1x_{j_2}^2\dots x_{j_{d-1}}^{d-1}{z}_{\vec{j}}(\tilde{\omega})\Big|\biggr)\nonumber
\end{eqnarray}
\begin{equation}
\leqslant 2\cdot \mathsf{E}_{\omega,\epsilon} \max_{x^i\in\mathcal{E}}\sum_{l=1}^n \epsilon_{l}\Big|\sum_{\vec{j}\in \mathbb{N}^{d}_{n}(d,l)}x_{j_1}^1x_{j_2}^2\dots x_{j_{d-1}}^{d-1}{z}_{\vec{j}}({\omega})\Big|.
\label{long ineq}
\end{equation}
%}

Next, since the function $\phi(u)=|u|$ is 1-Lipschitz on $\mathbb{R}$ and $\phi(0)=0$, by Talagrand's contraction inequality (see \cite[Formula (4.20) in the proof of Theorem 4.12]{LedTal} or \cite[Lemma 2.8]{APSS}), for a fixed $\omega\in[0,1]$ it follows that
$$
\mathsf{E}_{\epsilon} \max_{x^i\in\mathcal{E}}\sum_{l=1}^n \epsilon_{l}\Big|\sum_{\vec{j}\in \mathbb{N}^{d}_{n}(d,l)}x_{j_1}^1x_{j_2}^2\dots x_{j_{d-1}}^{d-1}{z}_{\vec{j}}({\omega})\Big|
\leqslant \mathsf{E}_{\epsilon} \max_{x^i\in\mathcal{E}}\sum_{l=1}^n \epsilon_{l}\sum_{\vec{j}\in \mathbb{N}^{d}_{n}(d,l)}x_{j_1}^1x_{j_2}^2\dots x_{j_{d-1}}^{d-1}{z}_{\vec{j}}({\omega}).$$
Taking then the expectation over $\omega\in [0,1]$ and applying Fubini theorem, we get
$$
\mathsf{E}_{\omega,\epsilon} \max_{x^i\in\mathcal{E}}\sum_{l=1}^n \epsilon_{l}\Big|\sum_{\vec{j}\in \mathbb{N}^{d}_{n}(d,l)}x_{j_1}^1x_{j_2}^2\dots x_{j_{d-1}}^{d-1}{z}_{\vec{j}}({\omega})\Big|
\leqslant \mathsf{E}_\epsilon\mathsf{E}_\omega \max_{x^i\in\mathcal{E}}\sum_{l=1}^n\sum_{\vec{j}\in \mathbb{N}^{d}_{n}(d,l)}\epsilon_{l}x_{j_1}^1x_{j_2}^2\dots x_{j_{d-1}}^{d-1}{z}_{\vec{j}}({\omega}).
$$
Observe now that the function $\psi_{l}(u)=\epsilon_{l} u$, $u\in \mathbb{R}$, is 1-Lipschitz for each $l=1,2,\dots,n$.
%a fixed collection of signs $\{\epsilon_{j_d}\}$, the functions $f_{\vec{j}}(t_{\vec{j}})=\epsilon_{j_d} t_{\vec{j}}$, $\vec{j}\in \mathbb{N}^{d}_{n}$ are 1-Lipschitz. 
Therefore, applying Talagrand's inequality once more (this time, to the expectation over $\omega$), we obtain
% and 1-Lipschitz functions $f_{\vec{j}}(t_{\vec{j}})=\epsilon_{j_d} t_{\vec{j}}$, $\vec{j}\in \mathbb{N}^{d}_{n}$. Then, 
$$
\mathsf{E}_\omega \max_{x^i\in\mathcal{E}}\sum_{\vec{j}\in \mathbb{N}^{d}_{n}(d,l)}\epsilon_{l}x_{j_1}^1x_{j_2}^2\dots x_{j_{d-1}}^{d-1}{z}_{\vec{j}}({\omega})\leqslant\mathsf{E}_\omega \max_{x^i\in\mathcal{E}}\sum_{\vec{j}\in \mathbb{N}^{d}_{n}(d,l)}x_{j_1}^1x_{j_2}^2\dots x_{j_{d-1}}^{d-1}{z}_{\vec{j}}({\omega}),
$$
and then from the preceding estimate it follows
\begin{eqnarray*}
\mathsf{E}_{\omega,\epsilon} \max_{x^i\in\mathcal{E}}\sum_{l=1}^n \epsilon_{l}\Big|\sum_{\vec{j}\in \mathbb{N}^{d}_{n}(d,l)}x_{j_1}^1x_{j_2}^2\dots x_{j_{d-1}}^{d-1}{z}_{\vec{j}}({\omega})\Big|
&\leqslant&
\mathsf{E}_\omega \max_{x^i\in\mathcal{E}}\sum_{l=1}^n\sum_{\vec{j}\in \mathbb{N}^{d}_{n}(d,l)}x_{j_1}^1x_{j_2}^2\dots x_{j_{d-1}}^{d-1}{z}_{\vec{j}}({\omega})\\
&=&
\mathsf{E}_\omega \max_{x^i\in\mathcal{E}}\sum_{\vec{j}\in \mathbb{N}^{d}_{n}}x_{j_1}^1x_{j_2}^2\dots x_{j_{d-1}}^{d-1}{z}_{\vec{j}}({\omega}).
\end{eqnarray*}

%\mathsf{E}_\omega \max_{x^i\in\mathcal{E},i\leqslant d-1}\sum_{\vec{j}\in \mathbb{N}^{d}_{n}}r_{\vec{j}} \epsilon_{j_d} a_{\vec{j}} x_{j_1}^1x_{j_2}^2\dots x_{j_{d-1}}^{d-1}
%$$
%$$
%\leqslant\mathsf{E}_\omega \max_{x^i\in\mathcal{E},i\leqslant d-1}\sum_{\vec{j}\in \mathbb{N}^{d}_{n}}r_{\vec{j}}  a_{\vec{j}}x_{j_1}^1x_{j_2}^2\dots x_{j_{d-1}}^{d-1}.$$
Thus, in view of \eqref{long ineq}, we conclude that
\begin{equation}
A_1\leqslant 2\cdot \mathsf{E}_\omega \max_{x^i\in\mathcal{E}}\sum_{\vec{j}\in \mathbb{N}^{d}_{n}}x_{j_1}^1x_{j_2}^2\dots x_{j_{d-1}}^{d-1}{z}_{\vec{j}}({\omega}).
\label{final second est}
\end{equation}
As a result, from \eqref{general est}, \eqref{general-1 est}  and \eqref{final second est} it follows 
\begin{eqnarray}
\mathsf{E}_\omega\|S_{n}(\cdot,\omega)\|_{L_\infty([0,1]^d)}&\leqslant& \sum_{l=1}^n\Big(\sum_{\vec{j}\in \mathbb{N}^{d}_{n}(d,l)}a_{\vec{j}}^2\Big)^{1/2}\nonumber\\
&+& 2\cdot \mathsf{E}_\omega \max_{x^i\in\mathcal{E}}\sum_{\vec{j}\in \mathbb{N}^{d}_{n}}x_{j_1}^1x_{j_2}^2\dots x_{j_{d-1}}^{d-1}{z}_{\vec{j}}({\omega}).
\label{final first step}
\end{eqnarray}

%$$
%\mathsf{E}_\omega \max_{x^i\in\mathcal{E},i\leqslant d-1}\sum_{l=1}^n\left(\Big|\sum_{\vec{j_d}'\in \mathbb{N}^{d-1}_{n}}z_{\vec{j}}x_{j_1}^1x_{j_2}^2\dots x_{j_{d-1}}^{d-1}\Big|-\mathsf{E}_\omega\Big|\sum_{\vec{j_d}'\in \mathbb{N}^{d-1}_{n}}z_{\vec{j}}x_{j_1}^1x_{j_2}^2\dots x_{j_{d-1}}^{d-1}\Big|\right)$$
%$$
%\leqslant 2\cdot \mathsf{E}_\omega \max_{x^i\in\mathcal{E},i\leqslant d-1}\sum_{\vec{j}\in \mathbb{N}^{d}_{n}}z_{\vec{j}} x_{j_1}^1x_{j_2}^2\dots x_{j_{d-1}}^{d-1}.$$\
Next, comparing relations \eqref{final second est} and
%$\mathsf{E}_\omega\|S_{a,n}(\vec{t},\omega)\|_{L_\infty([0,1]^d)}$
\eqref{EQ1}, we see that the same reasoning (with replacement of the vector $x^{d}$ with $x^{d-1}$) can be applied also to estimate the right-hand side of inequality  \eqref{final second est}. Then, precisely in the same way as before, we deduce that
\begin{eqnarray*}
\mathsf{E}_\omega \max_{x^i\in\mathcal{E}}\sum_{\vec{j}\in \mathbb{N}^{d}_{n}} x_{j_1}^1x_{j_2}^2\dots x_{j_{d-1}}^{d-1}{z}_{\vec{j}}({\omega})
&\leqslant&
\sum_{l=1}^n\Big(\sum_{\vec{j}\in \mathbb{N}^{d}_{n}(d-1,l)}a_{\vec{j}}^2\Big)^{1/2}\\
&+&2\cdot \mathsf{E}_\omega \max_{x^i\in\mathcal{E}}\sum_{\vec{j}\in \mathbb{N}^{d}_{n}}x_{j_1}^1x_{j_2}^2\dots x_{j_{d-2}}^{d-2}{z}_{\vec{j}}({\omega}).
\end{eqnarray*}
By combining this together with \eqref{final first step}, we obtain that
\begin{eqnarray*}
\mathsf{E}_\omega\|S_{n}(\cdot,\omega)\|_{L_\infty([0,1]^d)}&\leqslant& \sum_{l=1}^n\Big(\sum_{\vec{j}\in \mathbb{N}^{d}_{n}(d,l)}a_{\vec{j}}^2\Big)^{1/2}+ 2\cdot\sum_{l=1}^n\Big(\sum_{\vec{j}\in \mathbb{N}^{d}_{n}(d-1,l)}a_{\vec{j}}^2\Big)^{1/2}\\
&+& 2^2\cdot \mathsf{E}_\omega \max_{x^i\in\mathcal{E}}\sum_{\vec{j}\in \mathbb{N}^{d}_{n}}x_{j_1}^1x_{j_2}^2\dots x_{j_{d-2}}^{d-2}{z}_{\vec{j}}({\omega}).
\end{eqnarray*}
Proceeding in the same way, we get
\begin{eqnarray*}
\mathsf{E}_\omega\|S_{n}(\cdot,\omega)\|_{L_\infty([0,1]^d)}&\leqslant& \sum_{l=1}^n\Big(\sum_{\vec{j}\in \mathbb{N}^{d}_{n}(d,l)}a_{\vec{j}}^2\Big)^{1/2}+\dots + 2^{d-2}\cdot\sum_{l=1}^n\Big(\sum_{\vec{j}\in \mathbb{N}^{d}_{n}(2,l)}a_{\vec{j}}^2\Big)^{1/2}\\
&+& 2^{d-1}\cdot \mathsf{E}_\omega \max_{x^1\in\mathcal{E}}\sum_{\vec{j}\in \mathbb{N}^{d}_{n}}x_{j_1}^1{z}_{\vec{j}}({\omega}).
\end{eqnarray*}
Since
\begin{eqnarray*}
\mathsf{E}_\omega\max_{x^1\in\mathcal{E}}\sum_{\vec{j}\in \mathbb{N}^{d}_{n}} x_{j_1}^1{z}_{\vec{j}}({\omega}) =\mathsf{E}_\omega
\sum_{l=1}^n\Big|\sum_{\vec{j}\in \mathbb{N}^{d}_{n}(1,l)}a_{\vec{j}}r_{\vec{j}}(\omega)\Big|\leqslant
\sum_{l=1}^n\Big(\sum_{\vec{j}\in \mathbb{N}^{d}_{n}(1,l)}a_{\vec{j}}^2\Big)^{1/2},
\end{eqnarray*}
we arrive at estimate \eqref{upper est}, which completes the proof of the theorem.
%$$
%\mathsf{E}_\omega\|S_{a,n}(\vec{t},\omega)\|_{L_\infty([0,1]^d)}\leqslant \sum_{j_{d}=1}^n\Big(\sum_{\vec{j_d}'\in \mathbb{N}^{d-1}_{n}}a_{\vec{j}}^2\Big)^{1/2}+\dots + 2^{d-1}\cdot\sum_{j_{1}=1}^n\Big(\sum_{\vec{j}_1'\in \mathbb{N}^{d-1}_{n}}a_{\vec{j}}^2\Big)^{1/2},$$
%и оценка \eqref{upper est} доказана.
\end{proof}

Applying Theorem \ref{theor_mult_Rad}, we immediately obtain the following result. 

\begin{cor}
\label{Theorema_RUC_mult_Rad multiple}
The multiple Rademacher system of an arbitrary order $d$ possesses the RUC property  in the space $L_\infty([0,1]^d)$, that is, there is a constant $C_d>0$ such that for all $n\in\mathbb{N}$ and $a_{\vec{j}}\in\mathbb{R}$, $\vec{j}\in\mathbb{N}_n^d$, we have 
\begin{equation*}
\mathsf{E}_\theta\Big\|\sum_{\vec{j}\in \mathbb{N}^{d}_{n}}\theta_{\vec{j}}a_{\vec{j}}{\rm r}_{\vec{j}}^{\otimes}\Big\|_{L_\infty([0,1]^d)}\leqslant C_d\Big\|\sum_{\vec{j}\in \mathbb{N}^{d}_{n}}a_{\vec{j}}{\rm r}_{\vec{j}}^{\otimes}\Big\|_{L_\infty([0,1]^d)}.
\end{equation*}
%with some universal constant $C>0$ and all $n\in\mathbb{N}$ and $a_{\vec{j}}\in\mathbb{R}$, $\vec{j}\in\mathbb{N}_n^d$.
\end{cor}

Furthermore, in view of inequalities \eqref{eq_rad_cut_norm_mult} and definition \eqref{eq_rad_cut_norm_mult def} of the multidimensional cut-norm, the following assertion holds:

\begin{cor}
\label{cor_cut_norm multiple}
For every $d\in\mathbb{N}$, $n_1,n_2,\dots,n_d\in\mathbb{N}$ and all arrays $(a_{\vec{j}})_{j_k\leqslant n_k}$, with constants depending only on $d$, we have
$$
\mathsf{E}_{\theta}{\|(\theta_{\vec{j}}a_{\vec{j}})\|}_{cut}\asymp \min_{\theta_{\vec{j}}=\pm 1}{\|(\theta_{\vec{j}}a_{\vec{j}})\|}_{cut}\asymp \max_{1\le k\le d}\sum_{j_{k}=1}^{n_k} \Big(\sum_{j_1=1}^{n_1}\dots \sum_{j_{k-1}=1}^{n_{k-1}} \sum_{j_{k+1}=1}^{n_{k+1}}\dots \sum_{j_{d}=1}^{n_{d}}a_{{\vec{j}}}^2\Big)^{1/2}.
$$
\end{cor}

Corollary \ref{cor_cut_norm multiple} implies that the standard unit arrays
% $e^{k,l}=(e^{k,l}_{i,j})_{1\le i,j<\infty}$, $1\le k,l<\infty$, such that $e^{k,l}_{k,l}=1$   and $e^{k,l}_{i,j}=0$ if either $i\ne k$ or $j\ne l$, 
form a RUC sequence in the Banach space ${\mathcal M}_{cut}^d$ consisting of all infinite-dimensional arrays $A=(a_{\vec{j}})_{\vec{j}\in \mathbb{N}^d}$ such that
$$
\|A\|_{{\mathcal M}_{cut}^d}:=\sup_{n=1,2,\dots}{\|(a_{\vec{j}})_{j_k\le n}\|}_{cut}<\infty.$$

Now, precisely in the same way as in the case of the second-order Rademacher chaos  in Section \ref{Sec_chaos_Rad}, by using Theorem \ref{theor_mult_Rad} and decoupling  tools (see Corollary \ref{real theor_decoupling}), we can obtain the following results, where $\Delta^{d}_{n}:=\{ \vec{j}\in\Delta^d:\,j_d\leqslant n\}$.

\begin{cor}
\label{Theorema_RUC_mult_Rad chaos}
The Rademacher chaos of an arbitrary order $d$ has the RUC property  in the space $L_\infty$, that is, there is a constant $C_d'>0$ such that for all $n\in\mathbb{N}$ and $a_{\vec{j}}\in\mathbb{R}$, $\vec{j}\in\Delta_n^d$,   
\begin{equation*}
\mathsf{E}_{\theta}\Big\|\sum_{\vec{j}\in \Delta^{d}_{n}}\theta_{\vec{j}}a_{\vec{j}}{\rm r}_{\vec{j}}\Big\|_{L_\infty}\leqslant C_d'\Big\|\sum_{\vec{j}\in \Delta^{d}_{n}}a_{\vec{j}}{\rm r}_{\vec{j}}\Big\|_{L_\infty}
\end{equation*}
%with some universal constant $C>0$ and all $n\in\mathbb{N}$ and $a_{\vec{j}}\in\mathbb{R}$, $\vec{j}\in\Delta_n^d$.
\end{cor}

\begin{cor}
\label{cor_cut_norm chaos}
For every $d,n\in\mathbb{N}$, $d\le n$, and all strictly upper triangular arrays $(a_{\vec{j}})_{\vec{j}\in\Delta_n^d}$, with constants depending only on $d$, we have
$$
\mathsf{E}_{\theta}{\|(\theta_{\vec{j}}a_{\vec{j}})\|}_{cut}^*\asymp \min_{\theta_{\vec{j}}=\pm 1}{\|(\theta_{\vec{j}}a_{\vec{j}})\|}_{cut}^*
\asymp  \max_{1\le k\le d}\sum_{l=1}^{n} \Big(\sum_{\vec{j}\in\Delta^d_n\cap\mathbb{N}_n^d(k,l)}a_{{\vec{j}}}^2\Big)^{1/2},
$$
%{\color{blue}
%$$
%\mathsf{E}_{\theta}{\|(\theta_{\vec{j}}a_{\vec{j}})\|}_{\square}\asymp \min_{\theta_{\vec{j}}=\pm 1}{\|(\theta_{\vec{j}}a_{\vec{j}})\|}_{\square}\asymp  \max_{1\le k\le d}\sum_{j_{k}=1}^{n} \Big(\sum_{\vec{j_k}'\in\mathbb{N}^d_n:\,\vec{j}\in\Delta^d_n}a_{{\vec{j}}}^2\Big)^{1/2},
%$$
%}
\end{cor}

\section{Estimates for discrepancy of edge-weighted graphs and hypergraphs}
\label{Sec_disc_edges_graph}

Recall that an {\it edge-weighted graph} is a triple $G=(V, E, W)$, where $V$ is the set of {\it vertices}, $E\subset \{(v_1,v_2):\,v_1,v_2\in V,\,v_1\neq v_2\}$ is the set of {\it edges}, and $W=\{w(e)\}_{e\in E}$, $w(e)\in\mathbb{R}$, is the set of {\it weights}. For every function (coloring) $\theta:\,E\to\{-1,1\}$ we define the {\it discrepancy} of $G$ with respect to $\theta$ by
$$
\mathrm{disc}(G,\theta):=\max_{V'\subset V}\Bigl|\sum_{e=(v_1,v_2)\in E,\,v_i\in V'}\theta(e)w(e)\Bigr|.
$$
Then, the {\it discrepancy} of the graph $G$ is defined as follows:
$$
\mathrm{disc}(G)=\min_{\theta}\mathrm{disc}(G,\theta),
$$
where the minimum is taken over all colorings $\theta$. 

Problems related to the estimation of discrepancy and its analogues are very popular and important (see, for instance, \cite{AlonSpencer} or \cite{RCh}). Recall that historically one of the first results in this direction is the classical  theorem by Erd\"{o}s and Spencer about discrepancy of the complete graph $K_n$ with $n$ vertices in the unweighted case, i.e., when $w(e)=1$ for every $e\in E$. Namely, from a more general result, proved in \cite{ErSp} and related to homogeneous complete hypergraphs (see \eqref{Erd-Sp intr}), it follows that 
$$
\mathrm{disc}(K_n)\asymp n^{\frac{3}{2}},\;\;n\in\mathbb{N},
$$
with universal constants.
%Also, Erd\"{o}s and Spencer established similar result for homogeneous complete hypergraphs (see Section \ref{Intro}).

As we will see further, the results on the multiple Rademacher system and the Rademacher chaos, proved in the preceding sections, allow to extend estimates from \cite{ErSp} to arbitrary edge-weighted graphs and (homogeneous)  hypergraphs.

%obtain far-reaching development of those of the paper \cite{ErSp},  because they involve sharp (up to order) two-sided estimates for the discrepancy of edge-weighted graphs and hypergraphs.

(a) {\bf Bipartite complete graphs.} 
Let $K_{n,m}$ be the complete bipartite graph with the vertex set divided into two parts: $V=\{v_1,v_2,\ldots,v_n\}$ and $U=\{u_1,u_2,\ldots,u_m\}$. Then, the edge  set $E$ of $K_{n,m}$ consists of the pairs $(v,u)$, $v\in V$, $u\in U$. Let us assign to each edge $e=(v_i, u_j)\in E$ a weight $w(e)=a_{i,j}\in\mathbb{R}$, and consider the problem of estimation of the discrepancy of the edge-weighted graph 
%$\mathrm{disc}(G)$. 
$(K_{n,m},\{a_{i,j}\})$. 

%Thus, in our problem we are talking about an {\it edge} coloring of a graph with weights on the edges, i.e. about placing $\pm$-signs in front of the weights, and about finding (in the constructive version of the problem) such a set of these signs $\pm$, or, equivalently, choosing a number $\omega\in[0,1]$ that specifies the signs $r_{i,j}(\omega)$, so that for each bipartite subgraph $K_{I,J}$ of the graph $K_{n,m}$ the total weight of this subgraph, taking into account the placed signs,

%Denote by $K_{I,J}$ the complete bipartite graph with parts $I\subset V$ and $J\subset U$. 
Let $\theta= (\theta_{i,j})$ be an edge-coloring of $(K_{n,m},\{a_{i,j}\})$, that is, $\theta_{i,j}:=\theta(e)$ if $e=(v_i, u_j)$. Observe that to each set of vertices of this graph can be assigned a certain pair of sets $I\subset V$ and $J\subset U$. Consequently, 
%Since every subset of vertices of the graph $K_{n,m}$ corresponds to some set of edges $\{(v,u):\,v\in I,u\in J\}$ for appropriate $I\subset V$ and $J\subset U$, 
the discrepancy of the graph $(K_{n,m},\{a_{i,j}\})$ with respect to a coloring $\theta$ is defined by 
$$
\mathrm{disc}(K_{n,m},\{a_{i,j}\},\theta):=\max_{I\subset V,\,J\subset U}\Big|\sum_{i\in I,j\in J}\theta_{i,j} a_{i,j}\Big|.
$$
%Припишем каждому ребру $e=(v_i, u_j)\in E(K_{n,m})$ вес $w(e)=a_{i,j}\in\mathbb{R}$, и зададимся вопросом отыскания $\mathrm{disc}(G)$. Таким образом, в нашей задаче речь идет о {\it реберной} раскраске графа с весами на ребрах, т.е. о расстановке $\pm$-знаков перед весами, и о поиске (в конструктивном варианте задачи) такого набора этих знаков $\pm$, или, равносильно, выборе числа $\omega\in[0,1]$, задающего знаки $r_{i,j}(\omega)$, чтобы для каждого двудольного подграфа $K_{I,J}$ графа $K_{n,m}$ суммарный вес этого подграфа с учетом расставленных знаков
%$$
%W(K_{I,J}, \omega):=\Big|\sum_{i\in I,j\in J}r_{i,j}(\omega) a_{i,j}\Big|
%$$
%был ограничен сверху оптимальным (минимальным) числом. Это число, согласно определению, и обозначается 
%$\mathrm{disc}(K_{n,m})=\mathrm{disc}(K_{n,m}, \{a_{i,j}\})$.
Recalling definition \eqref{def_cut_norm}, we see that this is exactly the cut-norm of the matrix $\big(\theta_{i,j}a_{i,j}\big)_{i\leqslant n, j\leqslant m}$, and hence
$$
\mathrm{disc}(K_{n,m}, \{a_{i,j}\})=\min_{\theta_{i,j}=\pm 1}{\|(\theta_{i,j}a_{i,j})\|}_{cut}.
$$
This relation combined with Corollary \ref{cor_cut_norm} implies the following:

\begin{theorem} For every edge-weighted complete bipartite graph $(K_{n,m},\{a_{i,j}\})$, where $n,m\in\mathbb{N}$ and $a_{i,j}\in\mathbb{R}$, $1\le i\le m$, $1\le j\le n$, we have
\begin{eqnarray*}
\mathrm{disc}(K_{n,m}, \{a_{i,j}\}) &\asymp& \mathsf{E}_\theta \mathrm{disc}(K_{n,m},\{a_{i,j}\},\theta)\\
&\asymp&
\max\Big\{\sum_{i=1}^n\Big(\sum_{j=1}^m a_{i,j}^2\Big)^{1/2},\sum_{j=1}^m\Big(\sum_{i=1}^n a_{i,j}^2\Big)^{1/2}\Big\}
\end{eqnarray*}
with equivalence constants independent of $n,m$ and $\{a_{i,j}\}$.
\end{theorem}

(b) {\bf Edge-weighted graphs.} 
Assume first that $K_n$ is the complete graph with $n$ vertices and $a_{i,j}$ are arbitrary weights assigned to the edges $(v_i,v_j)$, $1\leqslant i<j\leqslant n$. In this case, the formula for the discrepancy with respect to an edge-coloring $\theta= (\theta_{i,j})_{1\le i<j\le n}$ takes the form:
$$
\mathrm{disc}(K_n,\{a_{i,j}\}_{i<j},\theta):=\max\Big\{\Big|\sum_{i,j\in I,\atop i<j}\theta_{i,j} a_{i,j}\Big|:\,I\subset [n]\Big\}
%\max_{V'\subset V}\Bigl|\sum_{e=(v_1,v_2)\in E,\,v_i\in V'}\theta(e)w(e)\Bigr|,
$$
and hence it coincides with the modified cut-norm \eqref{def_square} of the matrix $(\theta_{i,j} a_{i,j})_{1\le i<j\le n}$. Therefore, Corollary \ref{cor_square_norm} implies the following:
%$$
%\mathrm{disc}(K_{n}, \{a_{i,j}\}_{i<j}):=\min_{\omega\in[0,1]}\max\Big\{\Big|\sum_{i,j\in I,\atop i<j}r_{i,j}(\omega) a_{i,j}\Big|:\,I\subset\overline{[1,n]}\Big\}.
%$$
%Максимум под знаком минимума совпадает с $\square$-нормой. Поэтому, согласно %следствию \ref{cor_square_norm}, 
\begin{theorem} 
\label{complete graph}
Let $(K_{n},\{a_{i,j}\}_{i<j})$ be an edge-weighted complete graph, where $n\in\mathbb{N}$ and $a_{i,j}\in\mathbb{R}$, $1\le i<j\le n$, are arbitrary. Then,
\begin{eqnarray*}
\mathrm{disc}(K_{n}, \{a_{i,j}\}_{i<j}) &\asymp& \mathsf{E}_\theta \mathrm{disc}(K_{n},\{a_{i,j}\}_{i<j},\theta)\\
&\asymp&
\max\Big\{\sum_{i=1}^{n-1}\Big(\sum_{j=i+1}^n a_{i,j}^2\Big)^{1/2},\sum_{j=2}^n\Big(\sum_{i=1}^{j-1} a_{i,j}^2\Big)^{1/2}\Big\}
\end{eqnarray*}
with equivalence constants independent of $n$ and $\{a_{i,j}\}$.
\end{theorem}

%(c) {\bf General graphs.} 
Denoting the weights as $w(e)$, $e\in E$, rather than  $a_{i,j}$. $i<j$, one can readily check that
$$
\max\Big\{\sum_{i=1}^{n-1}\Big(\sum_{j=i+1}^n a_{i,j}^2\Big)^{1/2},\sum_{j=2}^n\Big(\sum_{i=1}^{j-1} a_{i,j}^2\Big)^{1/2}\Big\}\asymp \sum_{v\in V}\Bigl(\sum_{e\in E:\,v\in e}w(e)^2\Bigr)^{1/2}.
$$
Hence, the result of Theorem \ref{complete graph} can be re-written as follows:
%Taking into account equality \eqref{for general}, we can re-write the result of Theorem \ref{complete graph} as follows:
%Так как выражение справа эквивалентно правой части формулы \eqref{pre_decoupling} (см. рассуждение после этой формулы), то полученному выражению для $\mathrm{disc}(K_{n})$ можно придать следующую форму:
$$
\mathrm{disc}(K_{n}, \{w(e)\}_{e\in E})\asymp \sum_{v\in V}\Bigl(\sum_{e\in E:\,v\in e}w(e)^2\Bigr)^{1/2},
$$
where $V$ and $E$ are the sets of vertices and edges of the graph $K_n$, respectively. Since every edge-weighted graph with $n$ vertices can be obtained from $K_n$ by assigning zero weights to some edges, we arrive at the following general result which covers also the above-considered case of  bipartite complete graphs.
% Введем еще для произвольного графа с весами на ребрах $G=(V,E,W)$ и каждого фиксированного $\omega\in[0,1]$ величину
%$$
%W(G, \omega):=\Big|\sum_{e\in E}r_{i,j}(\omega) w(e)\Big|.
%$$
%В частности, такое обозначение будем использовать и для порожденных подграфов $G'=G[V']$ графа $G$, т.е. для графов $G'=(V',E', W')$ с условием $V'\subset V$, $E'= (V'\times V')\cap E$, $w'(e')=w(e')$ для $e'\in E'$. Теперь, используя следствие \ref{cor_square_norm}, мы можем сформулировать теорему об уклонении для произвольного графа с весами на ребрах.
\begin{theorem}\label{theor_weighted_graph} Let $G=(V,E,W)$ be an  arbitrary edge-weighted graph. Then, with universal constants we have
$$
\mathrm{disc}(G)\asymp \mathsf{E}_\theta \max_{V'\subset V}\Big|\sum_{e=(v_1,v_2)\in E, v_i\in V'}\theta(e)w(e)\Big|\asymp
 \sum_{v\in V}\Bigl(\sum_{e\in E:\,v\in e}w(e)^2\Bigr)^{1/2},
$$
where the expectation is taken over all colorings $\theta:\,E\to \{\pm 1\}$.
\end{theorem}
%\begin{theorem}\label{theor_weighted_graph} Пусть $G=(V,E,W)$ --- произвольный граф с весами $W=\{w(e)\}_{e\in E}$ на ребрах. Тогда
%$$
%\mathrm{disc}(G)\asymp \sum_{v\in V}\Bigl(\sum_{e\in E:\,v\in e}w(e)^2\Bigr)^{1/2} \asymp \mathsf{E}_\omega \max_{V'\subset V}W(G[V'],\omega)
%$$
%с константами эквивалентности, не зависящими от графа $G$.
%\end{theorem}

%В невесовом случае ($w(e)\equiv 1$) соответствующая задача для полного графа $K_n$ была решена Эрдешем и Спенсером \cite{ErSp}, которые показали, что для такого графа
%$$
%\mathrm{disc}(K_n)\asymp n^{\frac{3}{2}}.
%$$

(c) {\bf Homogeneous edge-weighted hypergraphs.} 

%One can easily see that Theorem \ref{theor_mult_Rad}, декаплинга и соотношения \eqref{eq_square_cut_norm_mult} imply sharp (up to constants) two-sided estimates for the discrepancy of weighted complete $d$-homogeneous hypergraphs. 

Let $n,d\in\mathbb{N}$, $2\le d\le n$ and let $H_{n,d}$ be the {\it complete $d$-homogeneous}  hypergraph with $n$ vertices, i.e., the edge set $E$ of $H_{n,d}$ consists of all subsets of the vertex set $V$ of cardinality $d$. 
%Let us introduce the discrepancy of $H_{n,d}$ by 
%$$
%\mathrm{disc}(H_{n,d})=\min_{\theta}\max_{V'\subset V}\Bigl|\sum_{e\in E,\,e\subset V'}\theta(e)\Bigr |,
%$$
%where the minimum is taking over all colorings $\theta:\,E\to\{-1,1\}$. As is already mentioned, in \cite{ErSp}, Erd\"os and Spencer proved the following estimates:
%\begin{equation}
%\label{Erd-Sp}
%c_d n^{\frac{d+1}{2}}\leqslant \mathrm{disc}(H_{n,d})\leqslant C_d n^{\frac{d+1}{2}},
%\end{equation}
%with some constants $c_d$ and $C_d$ independent of $n$.
Let us assign to each edge $e\in E$ of $H_{n,d}$ a weight $w(e)\in\mathbb{R}$. We define the discrepancy of resulting edge-weighted hypergraph $H_{n,d}(W)$ by the formula
$$
\mathrm{disc}(H_{n,d}(W)):=\min_{\theta}\max_{V'\subset V}\Bigl|\sum_{e\in E,\,e\subset V'}\theta(e)w(e)\Bigr|,
$$
where the minimum is taken over all edge-colorings $\theta$. Then, recalling the definition of the multi-dimensional modified cut-norm (see \eqref{quadrat_mult})
and applying Corollary \ref{cor_cut_norm chaos}, we obtain the following generalization of the above 
Erd\"os--Spencer result \cite{ErSp} (see also Section \ref{Intro}) to edge-weighted hypergraphs.

\begin{theorem}\label{theor_weighted_hypergraph} 
Let $n,d\in\mathbb{N}$, $2\le d\le n$. With some constants $c_d'$ and $C_d'$ independent of $n$ and the set of weights $W$, we have
\begin{eqnarray*}
c_d' \cdot \sum_{v\in V}\Bigl(\sum_{e\in E:\,v\in e}w(e)^2\Bigr)^{1/2}&\leqslant& \mathrm{disc}(H_{n,d}(W))\\&\leqslant& \mathsf{E}_\theta  \max_{V'\subset V}\Bigl|\sum_{e\in E,\,e\subset V'}\theta(e)w(e)\Bigr|\\
&\leqslant & C_d' \cdot\sum_{v\in V}\Bigl(\sum_{e\in E:\,v\in e}w(e)^2\Bigr)^{1/2}.
\end{eqnarray*}
%где математическое ожидание $\mathsf{E}$ берется по всем функциям $f\in F$, которые считаются равновероятными.
\end{theorem}

Observe that in the case when $w(e)\equiv 1$ we have
$$
\sum_{v\in V}\Bigl(\sum_{e\in E:\,v\in e}w(e)^2\Bigr)^{1/2}=n\cdot \sqrt{C_{n-1}^{d-1}}\asymp n^{\frac{d+1}{2}},
$$
with some constants depending only on $d$. Hence, by Theorem \ref{theor_weighted_hypergraph}, we obtain in particular estimates \eqref{Erd-Sp intr}, proved by Erd\"os and Spencer.

As above, this result can be extended immediately to arbitrary (not necessarily complete) homogeneous edge-weighted hypergraphs.

\end{document}